\documentclass[a4paper,12pt]{amsart}
\usepackage{latexsym,amscd,amssymb,url}
\usepackage[all,cmtip]{xy}  
\usepackage{graphicx}
\usepackage{xcolor}
\usepackage{enumerate} 
\definecolor{dblue}{rgb}{0,0,.6}
\usepackage[colorlinks=true, linkcolor=dblue, citecolor=dblue, filecolor = dblue, menucolor = dblue, urlcolor = dblue]{hyperref}
\usepackage{tikz}
\usetikzlibrary{cd}
\usepackage{mathtools}

\numberwithin{equation}{section}

\newtheorem{theorem}{Theorem}[section]

\theoremstyle{plain} %italic

\newtheorem{corollary}[theorem]{Corollary}
\newtheorem{definition}[theorem]{Definition}
\newtheorem{lemma}[theorem]{Lemma}
\newtheorem{proposition}[theorem]{Proposition}

\theoremstyle{definition} %roman

\newtheorem{example}[theorem]{Example}

\newtheorem{remark}[theorem]{Remark}

\newcommand{\Z}{\mathbb Z}

\newcommand{\A}{\mathbb A}

\newcommand{\C}{\mathbb C}
\newcommand{\N}{\mathbb N}

\newcommand{\CP}{\mathbb P}

\newcommand{\Tor}{\operatorname{Tor}}
\newcommand{\coker}{\operatorname{coker}}

\newcommand{\im}{\operatorname{im}}

\newcommand{\id}{\operatorname{id}}

\newcommand{\Spec}{\operatorname{Spec}}

\newcommand{\pr}{\operatorname{pr}}

\newcommand{\SING}{\operatorname{Sing}}

\newcommand{\CH}{\operatorname{CH}}
\newcommand{\supp}{\operatorname{supp}}
\newcommand{\sing}{\operatorname{sing}}

\newcommand{\Frac}{\operatorname{Frac}}

 \newcommand{\sm}{\operatorname{sm}}

\newcommand{\specialization}{\operatorname{sp}}
 
\newcommand{\dashedlongrightarrow}{\xymatrix@1@=15pt{\ar@{-->}[r]&}}
\renewcommand{\longrightarrow}{\xymatrix@1@=15pt{\ar[r]&}}
\renewcommand{\mapsto}{\xymatrix@1@=15pt{\ar@{|->}[r]&}}
\renewcommand{\twoheadrightarrow}{\xymatrix@1@=15pt{\ar@{->>}[r]&}}
\newcommand{\hooklongrightarrow}{\xymatrix@1@=15pt{\ar@{^(->}[r]&}}
\newcommand{\congpf}{\xymatrix@1@=15pt{\ar[r]^-\sim&}}
\renewcommand{\cong}{\simeq}

\begin{document}   
\title[On the rationality problem for low degree hypersurfaces]{On the rationality problem for low degree hypersurfaces}

\author{Jan Lange} 
\address{Institute of Algebraic Geometry, Leibniz University Hannover, Welfengarten 1, 30167 Hannover, Germany.}
\email{lange@math.uni-hannover.de}

\author{Stefan Schreieder} 
\address{Institute of Algebraic Geometry, Leibniz University Hannover, Welfengarten 1, 30167 Hannover, Germany.}
\email{schreieder@math.uni-hannover.de}

\date{\today}
%\date{June 18, 2024}
 %; secondary 14J45} 
%
% 14J70 Hypersurfaces
% 14J45 Fano varieties
% 14J10 Families, moduli, classification: algebraic theory
% 	14J35   	$4$-folds
% 14M20   	Rational and unirational varieties
% 	14M22   	Rationally connected varieties
% 14D06   	Fibrations, degenerations
%  	14E08   	Rationality questions
% 	14C25   	Algebraic cycles
%  	14M10   	Complete intersections
% 	14C30   	Transcendental methods, Hodge theory [See also 14D07, 32G20, 32J25, 32S35], Hodge conjecture


\begin{abstract}
We show that a very general hypersurface of degree $d \geq 4$ and dimension $N \leq (d+1)2^{d-4}$ over a field of characteristic $\neq 2$ does not admit a decomposition of the diagonal; hence, it is neither stably nor retract rational, nor $\mathbb{A}^1$-connected.  
    Similar results hold in characteristic $2$ under a slightly weaker degree bound.
    This improves earlier results in \cite{Sch-JAMS} and \cite{Moe}. 
\end{abstract}

\maketitle

\section{Introduction}

A variety $X$ over a field $k$ is retract rational if there is some integer $N\geq \dim X$ and rational maps $f\colon X\dashrightarrow \CP^N$ and $g\colon \CP^N\dashrightarrow X$ such that the composition $g\circ f$ is defined and agrees with $ \id_X$.
This notion is a direct analogue of retracts in topology; it was introduced into birational geometry by Saltman \cite{saltman,saltman-2} in the 1980s. 
Equivalently, $X$ is retract rational if there are an integer $N\geq \dim X$, open dense subsets $U\subset X$ and $V\subset \CP^N$, and a morphism $V\to U$ which is surjective on $L$-rational points for all field extensions $L/k$---this closely relates the notion to the classical question of parametrizing solutions of polynomial equations by rational functions. 
Rational or stably rational varieties are retract rational.

By the work of Asok--Morel \cite[Theorem 2.3.6]{asok-morel} and Kahn--Sujatha \cite[Theorem 8.5.1 and Proposition 8.6.2]{kahn-sujatha}, a smooth proper retract rational variety $X$ is $\mathbb{A}^1$-connected in the sense of $\mathbb{A}^1$-homotopy theory, i.e.\ $\pi_0^{\mathbb{A}^1}(X) = \{\ast \}$. 
By \cite[Remark 2.4.8]{asok-morel}, $\mathbb{A}^1$-connectedness is equivalent to separable $R$-triviality, which means that $X(L)/R=\{\ast\}$ for any separable field extension $L/k$, where $R$ denotes the equivalence relation on $X(L)$ generated by $x\sim y$ whenever $x,y\in X(L)$  lie on the same rational curve (defined over $L$).
In other words, $\mathbb{A}^1$-connectedness provides an arithmetic analogue of rational chain connectedness, 
which requires that any two $L$-rational points can be connected by a chain of rational curves for any algebraically closed field extension $L$ of $k$.
The following table illustrates the known implications for smooth proper varieties:
 $$
 \xymatrix{
\text{rational} \ar@{=>}[r]&   \text{stably rational}    \ar@{=>}[r]&  \text{retract rational} \ar@{=>}[d] 
 \ar@{=>}[r] &  \text{$\A^1$-connected}  \ar@{=>}[d] 
 \\ 
& & \text{unirational} \ar@{=>}[r] & \text{rationally chain connected}
 }
 $$
 The rationality problem for a given rationally connected variety $X$ asks `how rational it is', that is, which of the properties in the above diagram are satisfied.

Not every unirational variety is $\A^1$-connected \cite{artin-mumford} and not every stably rational variety is rational \cite{BCTSS}.
Moreover, there are retract rational varieties (over non-closed fields) that are not stably rational, see e.g.\ \cite[Theorem 1.5 and Theorem 2.3]{endo-miyata};
it is an open problem to produce such examples over algebraically closed fields. 
Whether any $\A^1$-connected variety is retract rational is open over any field. 

A smooth proper $\A^1$-connected variety  with a $k$-rational point has universally trivial Chow group of zero-cycles and hence admits a decomposition of the diagonal \cite{bloch-srinivas}, which is an interesting motivic and cycle-theoretic property in itself.

\subsection{Hypersurfaces}
A particularly interesting class of varieties for the rationality problem are smooth projective hypersurfaces $X\subset \CP^{N+1}_k$ of degree $d$ and dimension $N$ over a field $k$. 
The interesting range for the problem is when $d \leq N + 1$, in which case $X$ is Fano and thus rationally chain connected by \cite{campana,KMM}.

If $2^{d!}\leq N+1$ and $k=\C$, then $X$ is unirational, see \cite{BR,HMP}.  
If $d=N+1$ and $k=\C$, then $X$ is irrational (in fact birationally rigid) by a theorem of de Fernex \cite{deF1,deF2}, 
which extends earlier results by Iskovskikh--Manin \cite{IM} and Pukhlikov \cite{Pu1,Pu2}.
If $k=\C$ and $X\subset \CP^{N+1}_\C$ is very general of degree $d\geq 2\lceil \frac{N+3}{3} \rceil$, then it is not ruled and hence not rational by a theorem of Koll\'ar \cite{kollar}.
Under the slightly weaker bound $d\geq 2\lceil \frac{N+2}{3}\rceil$, Totaro \cite{totaro} showed that such hypersurfaces do not admit a decomposition of the diagonal, hence are neither stably nor retract rational, nor $\A^1$-connected. 
This used \cite{voisin,CTP}.
Totaro's result was improved in \cite{Sch-JAMS}, where the same result under the logarithmic bound $d\geq \log_2N+2$, $N\geq 3$ and over fields of characteristic $\neq 2$ was proven; a similar bound holds in characteristic 2 by \cite{Sch-torsion}.

The logarithmic degree bound in \cite{Sch-JAMS} is equivalent to $N\leq 2^{d-2}$. 
 In the case of stable rationality over fields of characteristic zero, Moe \cite{Moe} used the methods from \cite{NS,KT,NO} to improve this logarithmic bound  by a factor $(d+1)/4$ to cover the cases $N\leq (d+1)2^{d-4}$.
This paper generalizes Moe's result as follows:

\begin{theorem}\label{thm:main-irrationality-intro}
    Let $k$ be a field of characteristic different from $2$.
    Then a very general hypersurface $X\subset \CP^{N+1}_k$ of degree $d\geq 4$ and dimension $N\leq (d+1)2^{d-4}$ does not admit a decomposition of the diagonal, hence is neither stably nor retract rational, nor $\A^1$-connected.
\end{theorem}

While stable irrationality in characteristic zero follows in the above degree range from \cite[Theorem 5.2]{Moe}, the assertion on retract rationality and $\A^1$-connectedness are new.
In positive characteristic, for all $ N\leq (d+1)2^{d-4}$ not covered by \cite{Sch-JAMS}, even rationality was previously open.
For fixed degree $d$, a proportion of roughly $\frac{d-3}{d+1}$ cases are new.
The first new case concerns quintics of dimension $N=10$.

By a very general hypersurface $X$ over a field $k$ we mean one where the coefficients of a defining equation are algebraically independent over the prime field. 
With this definition, very general hypersurfaces exist over any field (not necessarily uncountable) of sufficiently large transcendence degree over the prime field, cf.\ Lemma \ref{lem:very-general} below.

In characteristic 2, we obtain an analogous result under a slightly weaker bound:

\begin{theorem}\label{thm:main-irrationality-intro-char2}
    Let $k$ be a field of characteristic $2$.
    Then a very general hypersurface $X\subset \CP^{N+1}_k$ of degree $d\geq 5$ and dimension $N\leq \frac{d+1}{3}2^{d-4}$ does not admit a decomposition of the diagonal, hence is neither stably nor retract rational, nor $\mathbb{A}^1$-connected.
\end{theorem}

In characteristic 2, the logarithmic bound in \cite{Sch-torsion} is given by $N\leq 2^{d-3}$.
The above theorem improves this by a factor $(d+1)/6$; for fixed $d$, the proportion of new cases is  given by $(d-5)/(d+1)$. 

Slightly better numerical bounds than in Theorems \ref{thm:main-irrationality-intro} and \ref{thm:main-irrationality-intro-char2} can be extracted from Theorem \ref{thm:main-torsion} (see also Theorem \ref{thm:main-torsion-order-intro}) below, which is our main result.

Our arguments allow us to bound the torsion order $\Tor(X)$ of the above hypersurfaces $X\subset \CP^{N+1}_k$, i.e.\ the smallest positive integer $e$ such that
$e\cdot \Delta_X$ decomposes in the Chow group of $X\times X$ (or $e=\infty$ if no such integer exists). 
If $\Tor(X)>1$, then $X$ does not admit a decomposition of the diagonal, hence is not $\A^1$-connected.
Moreover, any dominant generically finite map $f\colon \CP^{\dim X}\dashrightarrow X$ has degree $\deg f $ divisible by $\Tor(X)$ and so the torsion order yields an interesting lower bound on the possible degrees of unirational parametrizations of $X$.

If $X\subset \CP^{N+1}_k$ is a smooth Fano hypersurface of degree $d$ over some field $k$, then $\Tor(X)$ always divides $d!$, see  \cite{roitman} and  \cite[Proposition 5.2]{CL}.
 This yields an upper bound for the possible torsion orders of Fano hypersurfaces.
We then have the following result, which improves the previously known lower bounds from \cite{CL,Sch-torsion}. 

\begin{theorem}\label{thm:main-torsion-order-intro}
    Let $k$ be a field and let $m \geq 2$ be an integer invertible in $k$. 
    Let $n \geq 2$, $r \leq 2^n-2$, and $s \leq \left(\left\lfloor\frac{n}{m}\right\rfloor - 1 \right) (2^{n-1}-1)$ be non-negative integers and write $N\coloneq  n+r+s$.
    Then the torsion order of a very general Fano hypersurface $X \subset \CP^{N+1}_k$ of degree $d \geq m + n$ is divisible by $m$.
\end{theorem}

In Theorem \ref{thm:main-torsion} below we prove the above result under the weaker upper bound on $s$ given by
$ s \leq \sum\limits_{l = 1}^n \binom{n}{l} \left\lfloor \frac{n-l}{m} \right\rfloor$. 
 Previously, the best known bound on the torsion orders of hypersurfaces was contained in \cite{Sch-torsion} and  corresponds to the case $s=0$.

\subsection{Outline of the argument}

This paper provides a flexible cycle-theoretic analogue of the motivic obstruction from \cite{NS,KT}, which applies to degenerations into unions of varieties such that the obstruction lies in some lower-dimensional strata, and not in the components, as in \cite{voisin,CTP,Sch-Duke}. 

Previously, a solution to this problem has been proposed by the second named author with Pavic in \cite{PavicSch}, with an important recent generalization by the first named author in \cite{Lange}.
The main weakness of our previous approach is the fact that one has to compute an explicit strictly semi-stable model of the degeneration in question, control the combinatorics of the dual complex of the special fibre and control the Chow groups of 0- and 1-cycles of all components, which itself is a notoriously difficult task for almost any given smooth projective variety.
These tasks have been solved in a computationally involved manner for quartic fivefolds \cite{PavicSch} and $(3,3)$-complete intersections in $\CP^7$ \cite{LangeSkauli23}.
However, we do not see how to apply our obstruction from \cite{PavicSch,Lange} systematically to examples of higher dimensions, such as to the higher-dimensional complete intersections or hypersurfaces treated via the aforementioned motivic method in \cite{NO,Moe}; 
the total space of these degenerations are not strictly semi-stable (they have toric singularities) and the special fibre has a large number of components whose Chow groups seem inaccessible. 

The main improvement proposed in this paper is an extension of our previous method from \cite{PavicSch,Lange} to 
the non-proper case and hence in effect to 
(very) singular degenerations, that we could not handle before.
On a technical level, the idea is to work with pairs of a variety $X$ and a closed subset $W\subset X$.
 To state our obstruction, we say that a variety $X$ admits a decomposition of the diagonal with respect to a closed subset $W\subset X$, if the diagonal point $\delta_X\in \CH_0(X_{k(X)})$ lies in the image of $\CH_0(W_{k(X)})\to \CH_0(X_{k(X)})$, see Section \ref{sec:torsion-order} below. 
With this terminology, an ordinary decomposition of the diagonal corresponds to one with respect to a zero-dimensional closed subset.

The obstruction to rationality that we introduce and exploit reads as follows.

\begin{theorem}\label{thm:Obstruction-intro}
 Let $R$ be a discrete valuation ring with algebraically closed residue field $k$ and fraction field $K$. Let $\mathcal{X} \to \Spec R$ be a  proper flat %separated 
 $R$-scheme with geometrically integral generic fibre $X=\mathcal X\times_RK$ and special fibre $Y=\mathcal X\times_Rk$. 
    Let $W_{\mathcal{X}} \subset \mathcal{X}$ be a closed subscheme and let $W_X \coloneq  W_{\mathcal{X}} \cap X$ and $W_Y \coloneq  W_{\mathcal{X}} \cap Y$ be the intersections with the generic fibre and the special fibre, respectively. 
    Assume that the following conditions are satisfied:
    \begin{enumerate}[(1)]
        \item $\mathcal{X}^\circ \coloneq  \mathcal{X} \setminus W_{\mathcal{X}}$ is strictly semi-stable over $R$ (see Definition \ref{def:strictly-semi-stable} below); 
        \item $Y^\circ \coloneq  Y \setminus W_Y$ consists of two components $Y_0^\circ, Y_1^\circ$, with intersection $Z^\circ=Y_0^\circ\cap Y_1^\circ$.  
    \end{enumerate}
    If the geometric generic fibre $\bar X=X\times \bar K$ admits a decomposition of the diagonal relative to the closed subset $W_{\bar{X}} \coloneq  W_{X} \times_K \bar{K}$, then, for any field extension $L/k$, the map
    $$
        \Psi_{Y_L^\circ} \colon \CH_1(Y_0^\circ \times_k L) \oplus \CH_1(Y_1^\circ \times_k L) \longrightarrow \CH_0(Z^\circ \times_k L),\quad (\gamma_0,\gamma_1) \mapsto  
        \gamma_0|_{Z^\circ}-\gamma_1|_{Z^\circ}
    $$
    is surjective modulo any integer $m$ that is invertible in $k$, where 
 $\gamma_i|_{Z^\circ}$ denotes the pullback of $\gamma_i$ along the regular embedding $Z^\circ\hookrightarrow Y_i^\circ$, see \cite[Remark 2.3]{fulton}.  
\end{theorem}

Theorem \ref{thm:Obstruction-intro} admits a generalization to the case where $Y^\circ$ is an snc scheme without triple intersections, see Theorem \ref{thm:obstruction-proper} below.
Our arguments do not generalize to degenerations where the obstruction lies in deeper strata, see Remark \ref{rem:triple-intersection} below.

The obstruction map $\Psi_{Y_L^\circ}$ from Theorem \ref{thm:Obstruction-intro} is a refined version of the one in \cite{PavicSch,Lange}.
The presence of $W_{\mathcal X}$ in the above theorem yields the extra flexibility that was missing in \cite{PavicSch,Lange}.
In applications we will declare the singular locus of $\mathcal X$, the non-snc locus of $Y$, as well as all but two components of $Y$, to be contained in $W_{\mathcal X}$. 
In particular, the family $\mathcal X$ in the above theorem may be quite singular and the combinatorics of the special fibre $Y$ can be complicated.

To explain the mechanism of the above theorem, assume  that $Z^\circ$ is integral and let $Z\subset Y$ be the closure of $Z^\circ$.
Assume further that $\CH_1(Y_i^\circ \times_k L) = 0$ for $i=0,1$ and any field extension $L/k$.
(This will not hold on the nose in practice, but we will be able to achieve this after degeneration and show that this suffices for the argument.) 
In particular, the map $\Psi_{Y_L^\circ}$ in the above theorem is the zero map for every field extension $L/k$.  
Applying this to the function field $L=k(Z)$ of $Z$ and assuming that $\bar X$ admits a decomposition of the diagonal with respect to $W_{\bar X}$, 
we get that the image of the diagonal point $\delta_Z$ in $\CH_0(Z^\circ\times k(Z))$ vanishes, and so, by the localization sequence, $Z$ admits a decomposition with respect to $Z \cap W_Y$.
In other words,  $\bar X$ admits no decomposition of the diagonal with respect to $W_{\bar X}$ as long as  $Z$ admits no decomposition  with respect to $Z \cap W_Y$.
Examples of this strategy are illustrated in Examples \ref{ex:quartic-surface} and \ref{ex:quartic5folds} below. 

Since $\dim Z=\dim X-1$, the above reasoning sets the stage for an inductive argument where one increases the dimension by one in each step.
What makes this work is the observation that the examples of Fano hypersurfaces without a decomposition of the diagonal in \cite{Sch-JAMS,Sch-torsion} can in fact be shown to have no decomposition of the diagonal with respect to a large class of divisors, see Theorem \ref{thm:InductionStart} below.
This will serve as the start of our induction. 
For the induction step we degenerate a given hypersurface of degree $d$ to a union of two rational varieties which meet along the lower-dimensional hypersurface of degree $d$ that we have produced in the previous step of the induction, see Section \ref{section:doublecone} below for the precise degeneration we pick.
This step is inspired by \cite{Moe}.   
The total space of our degeneration as well as the fibres and their components will be (very) singular.
The singularities are not toric and so even in characteristic zero, the method in \cite{KT,NS,NO} does not seem to apply to our degeneration.

\begin{remark}
After completion of this paper, 
    James Hotchkiss and David Stapleton  informed us that they have independently obtained a different argument which shows that, over fields of characteristic zero, the hypersurfaces in Theorem \ref{thm:main-irrationality-intro} are not $\A^1$-connected and hence not retract rational,  see \cite{hotchkiss-stapleton}.
   Their approach relies on a homotopy-theoretical lift of the obstruction from \cite{NS,KT}. 
    As in \cite{NS,KT,NO,Moe}, the assumption on the characteristic is needed to be able to apply weak factorization.
\end{remark}

\section{Preliminaries}
\subsection{Conventions}
Rings are understood to be commutative with 1.
The characteristic of a ring $\Lambda$ is the smallest positive integer $c\in \Z_{\geq 1} $ such that any element in $\Lambda$ is $c$-torsion; it is zero if no such integer exists. 
The exponential characteristic of a field $k$ is $1$ if $k$ has characteristic zero and it is equal to the characteristic of $k$ otherwise.

An algebraic scheme is a separated scheme of finite type over a field.
A variety is an integral algebraic scheme.
Let $Y$ be an algebraic scheme, then we denote by $\CH_i(Y)$ the Chow groups of dimension $i$ cycles.
For a ring $\Lambda$, we let $\CH_i(Y,\Lambda)\coloneq \CH_i(Y) \otimes_\Z \Lambda$. 

Let $R$ be a ring. 
By an $R$-scheme we always mean a separated $R$-scheme of finite type, unless stated otherwise.
For an $R$-scheme $X$ and an $R$-algebra $A$, we denote the fibre product by $X \times_R A \coloneq  X \times_{\Spec R} \Spec A$ or simply by $X_A$.
We sometimes omit the ring $R$, if it is clear from context.

\subsection{Strictly semi-stable degenerations}

 \begin{definition} \label{def:snc}
    Let $k$ be a field. An snc scheme of dimension $n$ over $k$ is a geometrically reduced algebraic scheme $Y$ over $k$ with irreducible components $Y_i$, $i\in I$, such that for any subset $J\subset I$, the (scheme-theoretic) intersection $Y_J\coloneq \bigcap_{j\in J}Y_j$ is smooth over $k$ and,  if non-empty,  equidimensional of dimension $n+1-|J|$.
\end{definition}

We recall the definition of strictly semi-stable schemes over a discrete valuation ring, see e.g.\  \cite[Definition 1.1]{hartl}.

\begin{definition} \label{def:strictly-semi-stable}
Let $R$ be a discrete valuation ring with fraction field $K$ and residue field $k$.
A strictly semi-stable $R$-scheme $\mathcal X\to \Spec R$ is an irreducible, reduced, separated scheme which is flat and of finite type over $R$ with the following properties:
\begin{itemize}
    \item the generic fibre $X=\mathcal X\times_R K$ is smooth over $K$;
    \item the special fibre $Y=\mathcal X\times_R k$ is an snc scheme over $k$;
    \item each component of the special fibre $Y$ is a Cartier divisor on $\mathcal X$.
\end{itemize} 
\end{definition}

\subsection{Fulton's specialization map} \label{subsec:fulton}

Let $R$ be a discrete valuation ring with fraction field $K$ and residue field $k$.
Let $\mathcal X\to \Spec R$ be a flat $R$-scheme of finite type with generic fibre $X=\mathcal X\times_RK$ and special fibre $Y=\mathcal X\times_Rk$.
Then, for any ring $\Lambda$, there is a specialization map on Chow groups
$$
\specialization:\CH_i(X,\Lambda)\longrightarrow \CH_i(Y,\Lambda),
$$
defined as follows.
If  $\Lambda=\Z$ and $z=[Z]\in Z_i(X)$ is represented by an $i$-dimensional subvariety $Z\subset X$, then $\specialization(z)$ is represented by the restriction of the closure of $Z$ in $\mathcal X$ to $Y$.
(This restriction could be empty, in which case $\specialization(z)=[\emptyset]=0$.)
This extends $\Z$-linearly to a well-defined map by an argument of Fulton (see \cite[\S 4.4]{fulton75}, \cite[\S 20.3]{fulton}, or \cite[proof of Theorem 8.2]{Sch-survey}).
The case of arbitrary coefficients follows from this by functoriality of the tensor product.
If $k$ is algebraically closed, then the above map induces a well-defined map
$$
\specialization\colon \CH_i(\bar X,\Lambda)\longrightarrow \CH_i(Y,\Lambda) ,
$$
where $\bar X=X\times_K\bar K$ denotes the base change to an algebraic closure.
(In \cite[\S 4.4]{fulton75}, \cite[\S 20.3]{fulton}, this is shown for the completion $\hat R$ of $R$; the above case then follows via precomposing with the natural map $\CH_i(\bar X)\to \CH_i(X\times_K \overline{\Frac(\hat R)})$.)

We will need the following specific result on Fulton's specialization map.

 \begin{lemma}\label{lem:specialization-functoriality}
 Let $\Lambda$ be a ring and let $R$ be a discrete valuation ring with fraction field $K$ and residue field $k$.
Let $p:\mathcal X\to \Spec R$ and $q:\mathcal Y\to \Spec R$ be flat $R$-schemes of finite type.
Denote by $X_\eta,Y_\eta$ and $X_0,Y_0$ the generic and special fibres of $p$, $q$, respectively.
Assume that $Y_\eta$ is geometrically integral and that there is a geometrically integral component $Y_0'\subset Y_0$, such that 
 $A=\mathcal O_{\mathcal Y,Y'_0}$ is a discrete valuation ring and consider the flat $A$-scheme $\mathcal X_A\to \Spec A$, given by base change of $p$.
Then Fulton's specialization map induces a specialization map
$$
\specialization\colon\CH_i( X_\eta \times_{K} \bar{K}(Y_\eta),\Lambda )\longrightarrow \CH_i( X_0 \times_{k} \bar{k}(Y'_0),\Lambda ),
$$
where $\bar K$ and $\bar k$ denote algebraic closures of $K$ and $k$, respectively, such that the following holds:
\begin{enumerate}[(1)]
\item $\specialization$  commutes with pushforwards along proper maps and pullbacks along regular embeddings;
\item If $\mathcal X=\mathcal Y$ and $X_0$ is integral, then $\specialization(\delta_{X_\eta})=\delta_{X_0}$, where $\delta_{X_\eta}\in \CH_0( X_\eta \times_{K} \bar{K}(X_\eta),\Lambda)$ and $\delta_{X_0}\in \CH_0( X_0 \times_{k} \bar{k}(X_0),\Lambda)$ denote the diagonal points.
\end{enumerate} 
\end{lemma}
\begin{proof}
By functoriality of the tensor product, it suffices to prove the lemma in the case where $\Lambda=\Z$.
The lemma  is then stated under the assumption that $p$ and $q$ are proper with connected fibres  in \cite[Lemma 5.8]{PavicSch}, but the proof does not need those assumptions.
\end{proof}

\subsection{Very general hypersurfaces and their degenerations}
\begin{definition}\label{def:very-general}
    A hypersurface $X\subset \CP^{N+1}_k$ over a field $k$ is called very general, if the coefficients of a defining equation are algebraically independent over the prime field of $k$.
\end{definition}

We say that a variety $X$ over a field $L$ degenerates to a variety $Y$ over an algebraically closed field $k$ if there is a discrete valuation ring $R$ with residue field $k$ and fraction field $K$ with $K\subset L$ and a flat $R$-scheme $\mathcal X\to \Spec R$ of finite type whose special fibre is $Y$ and such that $\mathcal X\times _R L\cong X$.

\begin{lemma} \label{lem:very-general}
Let $\mathcal X\to B\coloneq \CP ( H^0(\CP^{N+1}_k,\mathcal O(d)))$ be the universal family of degree $d$ hypersurfaces of dimension $N$ over a field $k$.
Then the following hold:
\begin{enumerate}[(1)]
    \item The locus of very general hypersurfaces $B_{vg}\subset B$ (as a set of schematic points) is the complement of a countable union of closed subsets; it is non-empty if the transcendence degree of $k$ over the prime field is $\geq \dim B$.
    \item Let $X$ be a very general hypersurface, namely $X = X_b$, where $b \in B_{vg}$.
    There are (algebraically closed) field extensions $L/k(B)$ and $K/k$ together with an isomorphism of fields $\varphi:K\to L$, such that $\varphi$ induces an isomorphism of schemes
    $$
    X\times_kK\stackrel{\sim}\longrightarrow \mathcal X\times_{k(B)}L.
    $$
    In particular, up to a base change, $X$ degenerates to any other hypersurface $Y\subset \CP_k^{N+1}$ of degree $d$ in the above sense.
\end{enumerate}
\end{lemma}
\begin{proof}
We have $B=\CP^{N'}$ for some integer $N'$.
By definition, the  complement of $B_{vg}\subset B$ is given by the union of all hypersurfaces in $B=\CP^{N'}$ that are defined over the prime field $k_0$ of $k$.
This is a countable union and the complement contains a point as soon as $\operatorname{trdeg}_{k_0} k\geq \dim B$. This proves the first assertion.

Consider the universal family
$\mathcal X_0\to B_0=\CP ( H^0(\CP^{N+1}_{k_0},\mathcal O(d)))$ of degree $d$ hypersurfaces of dimension $N$ over the prime field $k_0$ of $k$.
The second item follows from the observation that any hypersurface of dimension $N$ and degree $d$ over a field extension of $k_0$, such that the coefficients of a defining equation are algebraically independent over $k_0$, is as an abstract scheme (i.e.\ without any structure morphism) a base change of the generic fibre of $\mathcal X_0\to B_0$.
This concludes the proof of the lemma.
\end{proof}

\section{Torsion orders and decompositions of the diagonal relative to a closed subset} \label{sec:torsion-order}

Let $X$ be a variety over a field $k$.
We say that $e$ times the diagonal of $X$ decomposes if there is a zero-cycle $z\in \CH_0(X)$ such that $$
e\cdot \Delta_X=z\times X+Z\in \CH_{\dim X}(X\times_k X),
$$
for some cycle $Z$ whose support does not dominate the second factor of $X\times_k X$, see \cite{bloch-srinivas}.
The torsion order of $X$, denoted by $\Tor(X)$, is the smallest positive integer $e$ such that a decomposition as above exists; it is $\infty$ if no such integer exists, see e.g.\ \cite{CL,kahn,Sch-torsion}.
We say that $X$ admits a decomposition of the diagonal if $\Tor(X)=1$.
By the localization exact sequence \cite[\S 1.8]{fulton}, this is equivalent to saying that
$ 
\delta_X\in \im(\CH_0(W_{k(X)})\to \CH_0(X_{k(X)}))
$ 
for a zero-dimensional closed subset $W\subset X$, where $\delta_X\in X_{k(X)}$ denotes the point induced by the diagonal $\Delta_X\subset X\times_k X$. 
This leads to the following simple but useful variant.

\begin{definition} \label{def:decomposition-relative-to-W}
Let $X$ be a variety over a field $k$ and let $\Lambda$ be a ring.
We say that $X$ admits a $\Lambda$-decomposition of the diagonal relative to a closed subset $W\subset X$ if
$$
\delta_X\in \im\left(\CH_0(W_{k(X)},\Lambda)\to \CH_0(X_{k(X)},\Lambda)\right) ,
$$
where $\delta_X$ denotes the diagonal point induced by the diagonal $\Delta_X \subset X \times_k X$.
If $\Lambda = \Z$, we also say that $X$ admits a decomposition (or integral decomposition) of the diagonal relative to $W$. 
\end{definition}

Variants of the notion of a decomposition of diagonal relative to a closed subscheme appear for example in \cite{bloch-srinivas}, \cite[Definition 1.2]{Voisin-CH0}, and \cite[Definition 1.1]{CL}.

\begin{remark}
By the above discussion, $X$ admits a decomposition of the diagonal if and only if it admits a decomposition relative to a closed subset of dimension zero.
\end{remark} 
 
By the localization sequence, the condition on $\delta_X$ in Definition \ref{def:decomposition-relative-to-W} is equivalent to
$$
\delta_X\in \ker\left(\CH_0(X_{k(X)},\Lambda)\to \CH_0(U_{k(X)},\Lambda)\right),
$$
where $U=X\setminus W$. 
It follows that a $\Lambda$-decomposition of the diagonal relative to $W\subset X$ is the same thing as a $\Lambda$-decomposition of the diagonal of $U$, relative to the empty set.
These observations lead  us to the following relative version of the aforementioned torsion order studied for instance in \cite{CL,kahn,Sch-torsion}.

\begin{definition} \label{def:torsion-order-relative-to-W}
    Let $X$ be a variety over a field $k$ and let $\Lambda$ be a ring.
    The $\Lambda$-torsion order of $X$ relative to a closed subset $W\subset X$, denoted by $\Tor^\Lambda(X,W)$, is the order of the element
    $$
    \delta_X|_U=\delta_U \in \CH_0(U_{k(X)},\Lambda) ,
    $$
    where $U=X\setminus W$.
\end{definition}

By definition, $\Tor^\Lambda(X,W)\in \N\cup \{\infty\}$.
If $\Lambda$ has characteristic $c\neq 0$, then $\Tor^\Lambda(X,W)$ divides $c$ and hence is finite.

We remark that the torsion order of proper varieties relative to the empty set has somewhat pathological behaviour. 
For instance, the $\Z$-torsion order of a proper variety with respect to the empty set is always $\infty$, but it may or may not be finite relative to a non-empty closed subset $W\subset X$.

\begin{remark} \label{rem:torsion-order-def}
The $\Lambda$-torsion order of $X$ relative to $W$ is $1$ if and only if $X$ admits a $\Lambda$-decomposition of the diagonal relative to $W$.
Moreover, the $\Lambda$-torsion order of $X$ relative to $W$ is nothing but the $\Lambda$-torsion order of $X\setminus W$, relative to the empty set.
\end{remark}
 
\begin{remark} \label{rem:torsion-order-X-reducible}
   If $X$ is an algebraic scheme over $k$ (not necessarily irreducible) and $W\subset X$ is a closed subset such that $U=X\setminus W$ is integral, then we can, in view of Remark \ref{rem:torsion-order-def}, still define $\Tor^\Lambda(X,W)$ via the order of the element 
    $$
    \delta_U \in \CH_0(U_{k(U)},\Lambda) .
    $$ 
\end{remark} 

\begin{lemma}\label{lem:torsion-order-observations}
Let $X$ be a variety over a field $k$ and let $W\subset X$ be closed.
Then the following hold:
\begin{enumerate}[(a)]
    \item For all $m \in \Z$, $\Tor^{\Z/m}(X,W) \mid \Tor^{\Z}(X,W)$. \label{item:torsion-order-Z/m-Z}
    \item Let $W' \subset W \subset X$ be a closed subset, then $\Tor^\Lambda(X,W) \mid \Tor^\Lambda(X,W')$. \label{item:torsion-order-decrease-for-larger-sets}
    \item $\Tor(X)$ is the minimum of the relative torsion orders $\Tor^{\Z}(X,W)$ where $W\subset X$ runs through all closed subsets of dimension zero.\label{item:3:torsion-order-observations}
    \item If $X$ is proper and $\deg\colon \CH_0(X) \to \Z$ is an isomorphism, then $\Tor(X)=\Tor^{\Z}(X,W)$ for any closed subset $W\subset X$ of dimension zero, which contains a zero-cycle of degree 1.\label{item:4:torsion-order-observations}
\item If $k=\bar k$ is algebraically closed, then $\Tor^\Lambda(X,W)=\Tor^\Lambda(X_L,W_L)$ for any ring $\Lambda$ and any field extension $L/k$.\label{item:0:torsion-order-observations}
\end{enumerate} 
\end{lemma}
\begin{proof}
Items \eqref{item:torsion-order-Z/m-Z}--\eqref{item:4:torsion-order-observations} follow easily from the definitions and the above discussions.
To prove item \eqref{item:0:torsion-order-observations}, note that $\Tor^\Lambda(X_L,W_L)\mid \Tor^\Lambda(X,W)$ (even without asking that $k$ is algebraically closed).
The converse divisibility statement follows via a straightforward  ``spreading out and specialization at a $k$-point'' argument.
This concludes the proof of the lemma.
\end{proof} 

The next lemma explains the geometric meaning of $\Lambda$-torsion orders.

\begin{lemma} \label{lem:weak-decomposiiton-of-diagonal}
Let $X$ be a proper variety over a field $k$ and let $\Lambda$ be a ring.
Assume that $X$ admits a resolution of singularities or that the exponential characteristic of $k$ is invertible in $\Lambda$. 
If for some closed subset $W\subset X$ the complement $U\coloneq X\setminus W$ is smooth, then 
$\CH_0(U_L,\Lambda)$
is $\Tor^\Lambda(X,W)$-torsion for all field extensions $L/k$.
\end{lemma} 
\begin{proof}
Since $\Tor^\Lambda(X_L,W_L)$ divides $\Tor^\Lambda(X,W)$,
we can assume without loss of generality that $L=k$. 
By work of Temkin \cite{temkin}, we can pick an alteration $\tau:X'\to X$ whose degree is a power of the exponential characteristic.
Moreover, we can choose $\tau$ to be of degree $1$ if $X$ admits a resolution of singularities.
We then let $W'\coloneq \tau^{-1}(W)$ and $U'=X'\setminus W'$.
We further let $V\subset U$ be the locus over which $\tau$ is \'etale and define $V'\coloneq \tau^{-1}(V)$.

Let $m\coloneq \Tor^\Lambda(X,W)$.
By assumptions, $m\cdot \delta_X\in \CH_0(X_{k(X)},\Lambda)$ vanishes when restricted to $U$. 
Hence,
$$
m\cdot \tau^\ast \delta_U=0\in \CH_0(U'_{k(U)},\Lambda).
$$
The base change of this class to the field extension $k(U')$ of $k(U)$ still vanishes.
If we spread this out and use the localization sequence, we find that
\begin{align} \label{eq:Gamma_tau}
m\cdot \Gamma = [Z_1]+[Z_2]\in \CH_{\dim X'}(X'\times X',\Lambda),
\end{align}
where $\Gamma\subset X'\times X'$ denotes the closure of the locus 
$$
(\tau|_{V'}\times \tau|_{V'})^{-1}(\Delta_V) = \{(x,y)\in V'\times V'\mid \tau(x)=\tau(y)\in V\}\subset X'\times X' ,
$$  
and $Z_1$, $Z_2$ denote some cycles with
$$
\supp Z_1\subset W'\times X'\ \ \text{and}\ \ \supp Z_2\subset X'\times D
$$ for some nowhere dense closed subset $D\subsetneq X'$.  

Let now $z\in \CH_0(U,\Lambda)$.
By Chow's moving lemma, we can assume that $\supp z\subset V$ and $\supp z\cap \tau(D)=\emptyset$.
We aim to show $m\cdot z=0\in \CH_0(U,\Lambda)$.
To this end, let $p:X'\times X'\to X'$ and $q:X'\times X'\to X'$ denote the projections to the first and second factors, respectively.  
Since $X'$ is smooth and proper over $k$, we can define pullbacks along correspondences in $\CH^{\dim X}(X'\times X',\Lambda)$.

By assumption, the support of $z$ lies in the locus over which $\tau$ is \'etale: $\supp z\subset V$.
We can thus define the zero-cycle  $\tau^\ast z$ on $X'$ on the level of cycles via the preimages of the points in the support of $z$.  
We then apply the correspondence $\Gamma$ and get 
$$
m\cdot \Gamma^\ast(\tau^\ast z)=p_\ast (m\Gamma \cdot q^\ast \tau^\ast z)\in \CH_0(X',\Lambda) .
$$ 
Since the support of $z$ is contained in the locus $V$
 over which $\tau$ is \'etale, a direct computation shows that the above zero-cycle has the property that
\begin{align} \label{eq:Gamma-lem-weak-dec}
 \tau_\ast (m\cdot \Gamma^\ast(\tau^\ast z))=(\deg \tau )^2 \cdot m\cdot  z\in \CH_0(X,\Lambda).
\end{align}
Conversely, by \eqref{eq:Gamma_tau}, we know that this cycle is rationally equivalent to
$$
\tau_\ast( \Gamma^\ast(m\cdot \tau^\ast z))= \tau_\ast ( [Z_1]^\ast(\tau^\ast z)+[Z_2]^\ast(\tau^\ast z))\in \CH_0(X,\Lambda).
$$
Since $ \supp Z_2\subset X'\times D$ and $\supp z\cap \tau(D)=\emptyset$, we have $[Z_2]^\ast(\tau^\ast z)=0$ and so
$$
\Gamma ^\ast(m\cdot \tau^\ast z)=[Z_1]^\ast(\tau^\ast z) \in \CH_0(X',\Lambda).
$$
Since $\supp Z_1\subset W'\times X'$, the above class lies in the image of $\CH_0(W',\Lambda)\to  \CH_0(X',\Lambda)$.
Hence,
$$
\tau_\ast \Gamma^\ast(m\cdot \tau^\ast z)\in \im(\CH_0(W,\Lambda)\to \CH_0(X,\Lambda)) .
$$ 
By \eqref{eq:Gamma-lem-weak-dec}, it follows that 
$$
(\deg \tau )^2 \cdot m\cdot  z \in \im (\CH_0(W,\Lambda)\longrightarrow \CH_0(X,\Lambda)) .
$$ 
By the localization sequence, this implies
$$
\deg(\tau)^2 \cdot m\cdot z|_U=0\in \CH_0(U,\Lambda) .
$$
Since $\deg(\tau)^2$ is either one or a power of the exponential characteristic of $k$, it is invertible in $\Lambda$ by assumption. 
We deduce, as we want, that $z \in \CH_0(U,\Lambda)$ is $m$-torsion. 
This concludes the proof of the lemma.
\end{proof}

We will use that 
$\Lambda$-torsion orders relative to closed subsets are well-behaved under specialization, as shown by the following lemma. 

\begin{lemma} \label{lem:Lambda-torsion-order-specialize}
Let $R$ be a discrete valuation ring with fraction field $K$ and  residue field $k$.
Let $\mathcal X\to \Spec R$ be a separated flat $R$-scheme of finite type with  generic fibre $X=\mathcal X\times  K$ and  special fibre $Y=\mathcal X\times k$.
Let $\bar X=X\times \bar K$ and $\bar Y=Y\times \bar k$ be the base changes to algebraic closures of $K$ and $k$, respectively. 
Let $W_{\mathcal X}\subset \mathcal X$ be a closed subset with $W_{\bar X}\coloneq W_{\mathcal X}\times \bar K\subset \bar X$ and $W_{\bar Y}\coloneq W_{\mathcal X}\times \bar k\subset \bar Y$.
Assume that the fibres of $U=\mathcal X\setminus W_{\mathcal X}$ over $R$  are non-empty and geometrically integral. 

Then we have
$$
\Tor^\Lambda(\bar Y,W_{\bar Y}) \mid \Tor^\Lambda(\bar X,W_{\bar X}) .
$$
\end{lemma}
\begin{proof}
By inflation of local rings \cite[Chapter IX, Appendice \S 2, Corollaire du Th\'eor\`eme 1 and Exercice 4]{bourbaki},
there is an unramified extension of discrete valuation rings $R'/R$ such that the residue field of $R'$ is $\bar k$.
Up to a base change along $R'/R$ we can thus assume that $k$
 is algebraically closed.
 (This uses item \eqref{item:0:torsion-order-observations} in Lemma \ref{lem:torsion-order-observations}.)

    Let $m\coloneq \Tor^\Lambda(\bar X,W_{\bar X})$.
    Then there is a finite extension $K'/K$ such that the $\Lambda$-torsion order of $X\times K'$ relative to $W_{X}\times K'$ is $m$.
    Let $R_{K'}\subset K'$ be the integral closure of $R$ in $K'$ and let $R'\subset K'$ be the localization of $R_{K'}$ at a maximal ideal lying over the maximal ideal of $R$.
    Then $R'$ is a discrete valuation ring with $\Frac R'=K'$ and $R\subset R'$.
    Up to a base change along $\Spec R'\to \Spec R$, we can then assume that $K=K'$.
    Hence, $m=\Tor^\Lambda( X,W_{X})$ and it remains to show that 
   \begin{align} \label{eq:m-divides-Tor(Y)}
\Tor^\Lambda( Y,W_{Y})\mid m  .
\end{align}

 Let $U=\mathcal X\setminus W_{\mathcal X}$ with generic fibre $U_\eta\coloneq U\times K$ and special fibre $U_0\coloneq U\times k$.
   By assumptions, $U\to \Spec R$ is flat with non-empty geometrically integral fibres. 
    Let $A$ be the local ring of $U$ at the generic point of the special fibre $U_0$.
    Since $U_0$ is integral, $A$ is a discrete valuation ring with residue field
    $k(U_0)$ and fraction field $K(U_\eta)$. 
    We then apply Fulton's specialization map on Chow groups to the base change $U_A\coloneq U\times_RA$ and obtain a group homomorphism
    $$
    \operatorname{sp}: \CH_0(U_\eta\times K(U_\eta),\Lambda)\longrightarrow \CH_0(U_0\times k(U_0),\Lambda)
    $$
    with $\operatorname{sp}(\delta_{U_\eta})=\delta_{U_0}$, see Lemma \ref{lem:specialization-functoriality}. 
    This implies \eqref{eq:m-divides-Tor(Y)}, because $m=\Tor^\Lambda( X,W_{X})$ is the order of $\delta_{U_\eta}$, while $\Tor^\Lambda( Y,W_{Y})$ is the order of $\delta_{U_0}$.
\end{proof}

\section{Obstruction map}

    Let $Y = \bigcup_{i \in I} Y_i$ be an snc scheme over $k$ (see Definition \ref{def:snc}).
    We say that $Y$ has no triple intersections if $Y_i \cap Y_j \cap Y_l = \emptyset$ for all pairwise different $i,j,l \in I$.
    
    The following map is the key player in Theorem \ref{thm:Obstruction-intro}, stated in the introduction.

    \begin{definition}
        Let $Y = \bigcup_{i \in I} Y_i$ be an snc scheme over $k$ which has no triple intersections. We fix a total order `$<$' on $I$. Let $\Lambda$ be a ring. Then we define the obstruction map
        \begin{equation}\label{eq:Psi}
            \Psi_Y^\Lambda \colon \bigoplus\limits_{l \in I} \CH_1(Y_l,\Lambda) \longrightarrow \bigoplus\limits_{\substack{i,j \in I \\ i < j}} \CH_0(Y_{ij},\Lambda), \quad (\gamma_l)_l \mapsto \left(\gamma_{i}|_{Y_{ij}} - \gamma_{j}|_{Y_{ij}} \right)_{i,j},
        \end{equation}
        where $Y_{ij} = Y_i \cap Y_j$ for $i,j \in I$.
    \end{definition}

In the above definition, $\gamma_i|_{Y_{ij}}$ is the intersection with the divisor $Y_{ij} \subset Y_i$, see \cite[\S 2.3]{fulton}. 
In particular, $\gamma_{i}|_{Y_{ij}}$ is represented by a $0$-cycle supported on the set-theoretic intersection $|\gamma_j|\cap Y_{ij}$ and hence can be viewed as a class in $\CH_0(Y_{ij},\Lambda)$, see also \cite[Convention 1.4]{fulton}.
We further note that we have $\CH_0(Y_{ij},\Lambda) = 0$ if $Y_{ij} = \emptyset$ and we set $\gamma_{i}|_{Y_{ij}}=0$ in this case.

    \begin{theorem}\label{thm:obstruction-general}
    Let $R$ be a discrete valuation ring with algebraically closed residue field $k$ and fraction field $K$. Let $\Lambda$ be a ring of positive characteristic $c \in \Z_{\geq 1}$ such that the exponential characteristic of $k$ is invertible in $\Lambda$. Let $\mathcal{X}  \to \Spec R$ be a strictly semi-stable $R$-scheme (see Definition \ref{def:snc}) with geometrically integral generic fibre $X  = \mathcal{X}  \times_R K$ and special fibre $Y  = \mathcal{X}  \times_R k$. Assume that $Y = \bigcup_{i \in I} Y_i$ has no triple intersections and fix a total order `$<$' on $I$. Then the cokernel of the map
    $$
        \Psi^\Lambda_{Y_L} \colon \bigoplus\limits_{l \in I} \CH_1(Y_l \times_k L,\Lambda) \longrightarrow \bigoplus\limits_{\substack{i,j \in I \\ i < j}} \CH_0(Y _{ij} \times_k L,\Lambda)
    $$
    from \eqref{eq:Psi} is $\Tor^\Lambda(\bar{X},\emptyset)$-torsion for every field extension $L/k$.
    \end{theorem}

    Note that the family $\mathcal{X} \to \Spec R$ in the above theorem is not assumed to be proper.
    We can always choose a relative Nagata compactification of this morphism, see  \cite[\href{https://stacks.math.columbia.edu/tag/0F41}{Tag 0F41}]{stacks-project}.
    Replacing $\Tor^\Lambda(\bar{X},\emptyset)$ by the relative torsion order of the compactification (cf.\ Definition \ref{def:torsion-order-relative-to-W} and Remark \ref{rem:torsion-order-def}), we can then rephrase the above theorem  as follows. 

    \begin{theorem}\label{thm:obstruction-proper}
        Let $R$ be a discrete valuation ring with algebraically closed residue field $k$ and fraction field $K$. Let $\Lambda$ be a ring of positive characteristic $c \in \Z_{\geq 1}$ such that the exponential characteristic of $k$ is invertible in $\Lambda$. Let $\mathcal{X} \to \Spec R$ be a proper flat %separated 
        $R$-scheme with geometrically integral generic fibre $X=\mathcal X\times_RK$ and special fibre $Y=\mathcal X\times_Rk$.
        Let $W_{\mathcal{X}} \subset \mathcal{X}$ be a closed subscheme and let $W_X \coloneq  W_{\mathcal{X}} \cap X$ and $W_Y \coloneq  W_{\mathcal{X}} \cap Y$ be the respective intersections (i.e.\ fibre products). 
        Assume that the following conditions are satisfied:
        \begin{enumerate}[(1)]
            \item $\mathcal{X}^\circ \coloneq  \mathcal{X} \setminus W_{\mathcal{X}}$ is a strictly semi-stable $R$-scheme (see Definition \ref{def:strictly-semi-stable}); \label{item:thm:obstruction-proper-strictlysemistable}
            \item $Y^\circ \coloneq  Y \setminus W_Y = \bigcup\limits_{i \in I} Y_i^\circ$ has no triple intersections. \label{item:thm:obstruction-proper-3intersection}
        \end{enumerate}
        Given a total order `$<$' on $I$, the cokernel of the map
        $$
            \Psi^\Lambda_{Y_L^\circ} \colon \bigoplus\limits_{l \in I} \CH_1(Y_l^\circ \times_k L,\Lambda) \longrightarrow \bigoplus\limits_{\substack{i,j \in I \\ i < j}} \CH_0(Y_{ij}^\circ \times_k L,\Lambda)
        $$
        from \eqref{eq:Psi} is $\Tor^\Lambda(\bar{X},W_{\bar X})$-torsion for every field extension $L/k$, where $W_{\bar{X}} \coloneq  W_{X} \times_K \bar{K}$.
\end{theorem}

\begin{corollary} \label{cor:Psi}
    Let the assumptions be as in Theorem \ref{thm:obstruction-proper}. 
    If $\bar{X}$ admits a $\Lambda$-decomposition of the diagonal relative to $W_{\bar{X}}$, then the map $$
        \Psi^\Lambda_{Y_L^\circ} \colon \bigoplus\limits_{l \in I} \CH_1(Y_l^\circ \times_k L,\Lambda) \longrightarrow \bigoplus\limits_{\substack{i,j \in I \\ i < j}} \CH_0(Y_{ij}^\circ \times_k L,\Lambda)
    $$
    is surjective for every field extension $L/k$.
\end{corollary}

We illustrate the mechanism of the above corollary in the following example, where we show formally that the fact that elliptic curves admit no decomposition of the diagonal with respect to a finite set of points implies that the same holds for certain quartic surfaces.
The result itself is of course a well-known consequence of \cite{bloch-srinivas} and the emphasis lies in the mechanism of the argument.

\begin{example} \label{ex:quartic-surface}
Let $\Lambda$ be a ring of characteristic $c\geq 2$.
Let $k$ be an algebraically closed field such that the exponential characteristic of $k$ is invertible in $\Lambda$ and let $f,g_0,g_1\in k[x_0,x_1,x_2,x_3]$ be general homogeneous polynomials of degrees $\deg f=4$ and $\deg g_i=2$.
Consider the degeneration $\mathcal X=\{t\cdot f+ g_0g_1=0\}$ over $R=k[[t]]$ of a smooth quartic surface into the union $Y_0\cup Y_1$ of two smooth quadric surfaces $Y_i=\{g_i=0\}$ which meet along a smooth elliptic curve $Z\coloneq Y_{01}=Y_0\cap Y_1$.
Let $S \subset \mathcal X$ be a closed reduced subscheme which is finite and surjective over $\Spec R$.
We further let $T_i\subset Y_i$ be the union of two lines contained in the two different rulings of the quadric surface $Y_i$ and set
$$
W_{\mathcal X}\coloneq S\cup T_0\cup T_1\cup \mathcal X^{\sing} ,
$$
where we note that the singular locus $\mathcal X^{\sing}=\{f=g_0=g_1=t=0\}$ is a finite number of points on $Z$.
In the above notation, $Y_i^\circ=Y_i\setminus W_{\mathcal X}$ is   an open subset of $Y_i\setminus T_i\cong \A^2$, hence  $\CH_1(Y_i^\circ \times_k L,\Lambda) = 0$ for $i=0,1$ and any field extension $L/k$.
By construction, $W_{\mathcal X}$ meets $Z$ in finitely many points.
Since the elliptic curve $Z$ does not admit a $\Lambda$-decomposition of the diagonal with respect to a finite collection of points, 
the group $\CH_0(Z^\circ\times_kL,\Lambda)$ is nontrivial for $L=k(Z)$, where $Z^\circ=Z\setminus W_{\mathcal X}$.
It follows that $\Psi^\Lambda_{Y^\circ_L}$ is not surjective and so Corollary \ref{cor:Psi} implies that the geometric generic fibre $\bar X$ of our family does not admit a $\Lambda$-decomposition of the diagonal with respect to the finite set of points $S\times_R \bar K$. 
\end{example}

The following example extends the argument in the previous example to the nontrivial case of quartic fivefolds, which was the main result in  \cite{PavicSch}.
We leave some details to the reader; similar constructions and arguments will be given with full details in the proof of Theorem \ref{thm:main-irrationality-intro} below.

 \begin{example}\label{ex:quartic5folds}
     Let $k$ be an algebraically closed field of characteristic different from $2$ and let $\Lambda=\Z/2$. 
     Consider the homogeneous polynomial
     $$
         f = x_0^2 y_1^2 + x_1x_2y_2^2 + x_0x_2y_3^2 + x_0x_1 g + x_0^3 z + x_0^2y_1 w\in k[x_0,x_1,x_2,y_1,y_2,y_3,z,w],
     $$ 
     where $g\coloneq x_0^2+x_1^2+x_2^2- 2x_0x_1-2x_0x_2 - 2x_1x_2$. 
     We consider the degeneration $$
         \mathcal{X} \coloneq  \{t x_0^2 + zw = f = 0\} \subset \CP^7_R \to \Spec R
     $$
     over $R = k[[t]]$ of a singular $(2,4)$-complete intersection $X$ in $\CP^7_{k((t))}$ into a union $Y = Y_0 \cup Y_1$ of two singular quartic fivefolds intersecting in the singular integral quartic fourfold
     $$
         Z \coloneq  Y_{01} = \{x_0^2 y_1^2 + x_1x_2y_2^2 + x_0x_2y_3^2 + x_0x_1g=0\} \subset \CP^5_{k},
     $$
     which is birational to the quadric surface bundle example in \cite[Example 8]{HPT}.
 Let $S\subset \mathcal X$ be the closure of a finite set of points in $X$ and let 
 $$
 W_\mathcal{X} \coloneq  S\cup \{x_0 x_1x_2 y_1 = 0\} \cup \SING Z \subset \mathcal{X}.
 $$
 One can check that $\mathcal{X}^\circ \coloneq  \mathcal{X} \setminus W_{\mathcal{X}}$ is strictly semi-stable.  
     The singular quartic fivefold $Y_i$ ($i=0,1$) is birational to $\CP^5$ via the projection from the singular point of multiplicity $3$ given by $x_i=0$ and $y_j = 0$ for all $i,j$ and $z=1$ (resp.\ $w=1$). 
     The restriction of this projection to the open subset $Y_i^\circ = Y_i \setminus W_\mathcal{X}$ gives an isomorphism with an open subset of $\mathbb{A}^5 = \CP^5 \setminus \{x_0 = 0\}$. 
     Hence, $\CH_1(Y_i^\circ \times_k L) = 0$ for every field extension $L/k$.
     Recall the unramified class $\alpha=(x_1/x_0,x_2/x_0)\in H^2_{nr}(k(Z),\Z/2)$ from \cite[Proposition 11]{HPT}. 
     A direct computation (using e.g.\ \cite[Theorem 2.3]{Sch-JAMS}) shows that this class vanishes on the generic points of the divisors $\{x_0 x_1x_2=0\}$ and $\{y_1=0\}$ (the latter uses that $\langle \frac{g}{x_0^2} \rangle$ is a subform of the quadratic form $\langle 1, \frac{x_1}{x_0} \rangle$ over $k(\CP^2)$).
     We then pass to a resolution or alteration of $Z$
 and apply the vanishing result in \cite[Proposition 5.1]{Sch-JAMS} to the unramified class from \cite[Proposition 11]{HPT}, to conclude via an ``action of correspondences'' argument that $Z$ does not admit a $\Lambda$-decomposition of the diagonal relative to $W_Z = W_\mathcal{X} \cap Z = (S \cap Z) \cup \{x_0 x_1 x_2 y_1 = 0\} \cup \SING Z \subset Z$; see Theorem \ref{thm:InductionStart} below for more details on this type of argument.
     Thus, $\CH_0(Z^\circ \times k(Z),\Lambda)$ is nontrivial, where $Z^\circ=Z\setminus W_Z$.
     It follows that the map $\Psi^\Lambda_{Y^\circ_{k(Z)}}$ in Corollary \ref{cor:Psi} is not surjective.
     Consequently, $\bar{X}$ does not admit a $\Lambda$-decomposition of the diagonal relative to $W_{\bar X} = (W_\mathcal{X} \cap X)_{\bar K}$.
     The substitution $y_1=zy_1/x_0$ and $w=-tx_0^2/z $ shows that
     %Since 
     the $(2,4)$ complete intersection $X$ is birational to the singular quartic hypersurface
     $$
         X' = \{z^2 y_1^2 + x_1x_2y_2^2 + x_0x_2y_3^2 + x_0x_1g + x_0^3 z - tx_0^3y_1  = 0\} \subset \CP^6_{k((t))}.
     $$ 
     One checks that the restriction of the induced birational map $X \dashrightarrow X'$ to $X \setminus W_\mathcal{X}$ yields an isomorphism onto its image. 
     It follows that, even after base change to an algebraic closure of $k((t))$, $X'$ does not admit a decomposition of the diagonal relative to the union of $\{x_0x_1x_2 y_1 z = 0\} \subset X'$ with a finite set of points.
%     One checks that the restriction of the induced birational map $X\dashrightarrow X'$ to $X \setminus W_\mathcal{X}$ yields an isomorphism onto its image, it follows that even after base change to an algebraic closure of $k((t))$, $X'$ does not admit a decomposition of the diagonal relative to the union of $\{x_0x_1x_2 y_1 z = 0\} \subset X'$ with a finite set of points.  
     This produces a singular quartic fivefold that does not admit a $\Lambda$-decomposition of the diagonal relative to a finite set of points (union its singular locus).
     We can then apply Lemma \ref{lem:Lambda-torsion-order-specialize} to conclude that any smooth quartic fivefold that degenerates to $X'$ does not admit a decomposition of the diagonal relative to a finite set of points.
     Altogether, this yields a simplified argument for the main result in \cite{PavicSch}. 
 \end{example}

 \begin{remark}
    In order to prove Theorem \ref{thm:main-irrationality-intro}, we would like to replace in the above argument the example \cite[Example 8]{HPT} with the higher dimensional examples in \cite[Section 4 and 7]{Sch-JAMS}. 
    This does not work directly as the unramified class $\alpha$ from \cite[Proposition 5.1]{Sch-JAMS} does not vanish along the divisor $\{y_1 = 0\}$. %along the divisor $\{y_1 = 0\}$ does not vanish in loc.\ cit.. 
    We circumvent this problem by introducing an additional parameter $\lambda$ in the family, see Section \ref{section:doublecone}.
    Further technical difficulties appear because we aim to apply the above argument inductively, where we degenerate to $Y_0\cup Y_1$ such that $Y_0\cap Y_1$ is birational to the example of lower dimension we had treated before.
 \end{remark}

\subsection{Comparison to the map in \cite{PavicSch}}

We recall the obstruction map from \cite{PavicSch} in the generality needed in this paper.
Let $\Lambda$ be a ring and let $Y = \bigcup\limits_{i \in I} Y_i$ be an snc scheme over a field $k$. Then we consider for $i,j \in I$ the homomorphism
\begin{equation}\label{eq:Phi}
    \Phi^\Lambda_{Y_i,Y_j} \colon \CH_1(Y_i,\Lambda) \longrightarrow \CH_0(Y_j,\Lambda), \quad \gamma_i \mapsto \begin{cases}
        \iota_{ij,j\, \ast} (\gamma_i|_{Y_{ij}}) & \text{if } i \neq j, \\
        - \sum\limits_{l \neq i} \iota_{il,i\, \ast} (\gamma_i|_{Y_{il}}) & \text{if } i = j,
    \end{cases}
\end{equation}
where $\iota_{ij,i} \colon Y_{ij} \coloneq  Y_i \cap Y_j \to Y_i$ is the natural inclusion and $\gamma_i|_{Y_{ij}} \coloneq  \iota_{ij,i}^\ast \gamma$ is the pullback along the regular embedding $\iota_{ij,i}$ of the Cartier divisor $Y_{ij}$ in $Y_i$, 
see \cite[Definition 3.1 and Lemma 3.2]{PavicSch}. 
We additionally set $$
        \Phi^\Lambda_{Y,Y_j} \coloneq  \sum\limits_{i \in I} \Phi^\Lambda_{Y_i,Y_j} \colon \bigoplus\limits_{i \in I} \CH_1(Y_i,\Lambda) \longrightarrow \CH_0(Y_j,\Lambda)
$$
for $j \in I$ and $$
    \Phi^\Lambda_{Y} \coloneq  \sum\limits_{j \in I} \Phi^\Lambda_{Y,Y_j} \colon \bigoplus\limits_{i \in I} \CH_1(Y_i,\Lambda) \longrightarrow \bigoplus\limits_{j \in I} \CH_0(Y_j,\Lambda).
$$

\begin{lemma}\label{lem:alternativePhi}
    Let $R$ be a discrete valuation ring with residue field $k$ and let $\mathcal{X} \to \Spec R$ be a strictly semi-stable $R$-scheme with special fibre $Y \coloneq  \mathcal{X} \times_R k$. 
    Then for $i,j \in I$ and  $\gamma_i \in \CH_1(Y_i,\Lambda)$, we have
    $$ 
        \Phi^\Lambda_{Y_i,Y_j}(\gamma_i) = \iota_j^\ast \iota_{i \, \ast} \gamma_i \in \CH_0(Y_j,\Lambda),
    $$ 
    where $\iota_i \colon Y_i \to \mathcal{X}$ is the natural inclusion.
\end{lemma}
\begin{proof}
    This follows from \cite[Theorem 6.2]{fulton}, see also \cite[Lemma 3.2]{PavicSch}.
\end{proof}

The effect of base changes for $\Phi$ has previously been studied in \cite[Section 4]{PavicSch} and \cite[Theorem 1.1]{Lange}; some version of $\Psi$ appeared implicitly in the proof of \cite[Proposition 4.4]{PavicSch}.

Proposition \ref{prop:basechange+resolution-general} below compares $\Phi$ and $\Psi$ and shows that the morphism $\Psi$ introduced in \eqref{eq:Psi} has excellent properties under base changes.
This will play an important role in the proof of Theorem \ref{thm:obstruction-general} below.

Recall that the dual graph $G = (V,E)$ of an snc scheme $Y$ without triple intersections is the graph whose vertices $V$ correspond to the irreducible components of $Y$ and whose edges $E$ encode the codimension $1$ subvarieties of $Y$. Moreover, we explicitly fix a total order `$<$' on the set of vertices $V$.
We denote for a vertex $v \in V$ and an edge $e \in E$ the corresponding irreducible subvariety by $Y_v$ and $Y_e$, respectively.
For an edge $e \in E$, we denote its end points by $v(e),w(e) \in V$ and assume that $v(e) < w(e)$. (Note that every edge in $G$ has two distinct vertices.)

\begin{lemma} \label{lem:basechange+resolution-general}
        Let $\Lambda$ be a ring of finite characteristic $c \in \Z_{\geq 1}$. Let $R$ be a discrete valuation ring with fraction field $K$ and algebraically closed residue field $k$. Let $\mathcal{X} \to \Spec R$ be a strictly semi-stable $R$-scheme whose special fibre $Y \coloneq  \mathcal{X} \times_R k$ has no triple intersections with dual graph $G = (V,E)$. 
    Let $R'/R$ be a finite ramified extension of discrete valuation rings such that the ramification index $r$ is divisible by $c$. Let $\mathcal{X}' \to \mathcal{X} \times_R R'$ be a resolution given by repeatedly blowing up the non-Cartier components of the special fibre as in \cite{hartl}. 
    Then the following holds:
    \begin{enumerate}[(1)]
        \item $\mathcal{X}'$ is strictly semi-stable and its special fibre $Y'$ has no triple intersections. 
        The dual graph $G' = (V',E')$ of $Y'$ consists of the vertices $V' = V \cup (E \times \{1,\dots,r-1\})$ and edges $(v',w') \in V' \times V'$ of the form $$
        \begin{aligned}
            v' &= (e,n), &&w'= (e,n+1) &&\text{for } e \in E,\ n \in \{1,\dots,r-2\}, \\
            v' &= v(e), &&w'= (e,1) &&\text{for } e \in E, \text{ or} \\
            v' &= (e,r-1), &&w'= w(e) &&\text{for } e \in E.
        \end{aligned}
        $$ \label{item:prop:basechange+res-general:dualgraph}
        \item The components $Y'_v$ for $v \in V$ are isomorphic to $Y_v$ and the components $Y'_{(e,n)}$ for $e \in E$ and $1 \leq n < r$ are $\CP^1$-bundles over $Y_{e}$ with two disjoint sections \begin{align*}
            &s_{(e,n)} \colon Y'_{(e,n)} \cap Y'_{(e,n+1)} \longrightarrow Y'_{(e,n)}, \\
            &s'_{(e,n)} \colon Y'_{(e,n-1)} \cap Y'_{(e,n)} \longrightarrow Y'_{(e,n)},
        \end{align*}
        given by the natural inclusion, where we set $Y'_{(e,0)} = Y'_{v(e)}$ and $Y'_{(e,r)} = Y'_{w(e)}$.
        \label{item:prop:basechange+res-general:components}
    \end{enumerate}
\end{lemma}
\begin{proof}
        The geometry of the resolution after finite base-change is explained for example in \cite[Section 4.2]{PavicSch} for chains.
        The same arguments work for any special fibre that has no triple intersections, see also \cite[Proposition 4.11]{Lange}.
\end{proof}

\begin{remark}
    The graph $(V',E')$ is obtained by subdividing each edge of $(V,E)$ into $r$ pieces, see for example \cite[Example A.1]{Lange}.
\end{remark}

\begin{proposition}\label{prop:basechange+resolution-general}
In the notation of Lemma \ref{lem:basechange+resolution-general},
let $L/k$ be a field extension. If the cokernel of $$
        \Phi^{\Lambda}_{Y'_L} \colon \bigoplus\limits_{v' \in V'} \CH_1(Y'_{v'} \times_k L,\Lambda) \longrightarrow  \bigoplus\limits_{w' \in V'} \CH_0(Y'_{w'} \times_k L,\Lambda)
        $$
        is $m$-torsion for some integer $m$, then the cokernel of the map
        $$
            \Psi^\Lambda_{Y_L} \colon \bigoplus\limits_{l \in V} \CH_1(Y_l \times_k L,\Lambda) \longrightarrow \bigoplus\limits_{\substack{i,j \in V \\ i < j}} \CH_0(Y_{ij} \times_k L,\Lambda)
        $$
        from \eqref{eq:Psi} is $m$-torsion. 
\end{proposition}
\begin{proof}
    For ease of notation, we will deal with the case $L=k$ in what follows; the general case follows verbatim via the same argument.

    Recall that $q_{(e,n)}:Y'_{(e,n)} \to Y_{e}$ is a $\CP^1$-bundle for $1 \leq n < r$ by item \eqref{item:prop:basechange+res-general:components} in Lemma \ref{lem:basechange+resolution-general}. 
    Thus there exists isomorphisms \begin{align}
        \CH_1(Y'_{(e,n)},\Lambda) &\cong \CH_1(Y_{e},\Lambda) \oplus q_{(e,n)}^\ast \CH_0(Y_{e},\Lambda), \label{eq:projective-bundle-CH1}\\
        \CH_0(Y'_{(e,n)},\Lambda) &\cong \CH_0(Y_{e},\Lambda), \label{eq:projective-bundle-CH0}
    \end{align}
    see \cite[Theorem 3.3 (b)]{fulton}.
    Note that the subspace $q_{(e,n)}^\ast \CH_0(Y_e,\Lambda)$ is canonical, while the subspace $\CH_1(Y_{e},\Lambda)\subset \CH_1(Y'_{(e,n)},\Lambda)$ depends on a choice of a section of $q_{(e,n)}:Y'_{(e,n)} \to Y_{e}$.

    \textbf{Step 1.} We will show that for every $\gamma' \in \bigoplus_{v' \in V'} \CH_1(Y'_{v'},\Lambda)$, there exists another class $\gamma  \in \bigoplus_{v' \in V'} \CH_1(Y'_{v'},\Lambda)$ with the same image $\Phi^\Lambda_{Y'}(\gamma) = \Phi^\Lambda_{Y'}(\gamma')$, such that the component $\gamma_{(e,n)}\in \CH_1(Y'_{(e,n)},\Lambda)$ of $\gamma$ satisfies
    \begin{equation}\label{eq:gamma-assumption}
        \gamma_{(e,n)} \in q^\ast_{(e,n)} \CH_0(Y_e,\Lambda) \subset \CH_1(Y'_{(e,n)},\Lambda)
    \end{equation}
    for each $(e,n) \in E \times \{1,\dots,r-1\}$.

    This follows from the argument in \cite[Lemma 4.3]{PavicSch}, which we explain for the convenience of the reader. Let $\gamma' =(\gamma'_{v'})_{v'} \in \bigoplus_{v' \in V'} \CH_1(Y'_{v'},\Lambda)$ be a collection of one-cycles and let $(e,n) \in E \times \{1,\dots,r-1\}$. Using \eqref{eq:projective-bundle-CH1}, we can write $\gamma'_{(e,n)}$ as
    $$
        \gamma'_{(e,n)} = q_{(e,n)}^\ast \alpha_{(e,n)} + s_{(e,n) \, \ast} \zeta_{(e,n)},
    $$
    for some $\alpha_{(e,n)} \in \CH_0(Y_e,\Lambda)$ and some $\zeta_{(e,n)} \in \CH_1(Y_e,\Lambda)$, where $s_{(e,n)}$ is the section given in item \eqref{item:prop:basechange+res-general:components} of Lemma \ref{lem:basechange+resolution-general}.
    Lemma \ref{lem:alternativePhi} implies that
    \begin{equation*}
        \Phi_{Y'}^\Lambda(s_{(e,n) \, \ast} \zeta_{(e,n)}) = \Phi_{Y'}^\Lambda(s'_{(e,n+1) \, \ast} \zeta_{(e,n)}),
    \end{equation*}
    where $s'_{(e,n+1)} \colon Y'_{(e,n)} \cap Y'_{(e,n+1)} \to Y'_{(e,n+1)}$ is the natural inclusion. Thus $$
        \Phi^\Lambda_{Y'}(\gamma') = \Phi^\Lambda_{Y'}(\gamma''),
    $$
    where $\gamma'' = (\gamma''_{v'})_{v'} \in \bigoplus_{v' \in V'} \CH_1(Y'_{v'},\Lambda)$ is given by
    $$
        \gamma''_{v'} = \begin{cases}
            q_{(e,n)}^\ast \alpha_{(e,n)} & \text{if } v' = (e,n), \\
            \gamma'_{(e,n+1)} + s'_{(e,n+1) \, \ast} \zeta_{(e,n)} & \text{if } v' = (e,n+1),  \\
            \gamma'_v & \text{otherwise,}
        \end{cases}
    $$
    where we set $\gamma'_{(e,r)} = \gamma'_{w(e)} \in \CH_1(Y'_{w(e)},\Lambda)$. Applying this argument for every edge $e \in E$ and $1 \leq n \leq r-1$ (in increasing order) finishes Step 1.

     Since $Y$ is an snc scheme without triple intersections, we have $Y_e \cap Y_{e'} = \emptyset$ for all different $e,e' \in E$. 
     In particular, \begin{equation*}
        \bigoplus\limits_{\substack{i,j \in V \\ \substack i < j}} \CH_0(Y_{ij},\Lambda) = \bigoplus\limits_{e \in E} \CH_0(Y_e,\Lambda).
    \end{equation*}
    Using this identification, there is a natural projection homomorphism $$
        \pr_e \colon \bigoplus\limits_{\substack{i,j \in V \\ \substack i < j}} \CH_0(Y_{ij},\Lambda) \longrightarrow \CH_0(Y_e,\Lambda)
    $$
    for every $e \in E$. We denote the composition $\pr_e \circ \Psi_Y^\Lambda$ by $\Psi^\Lambda_{Y,Y_e}$.

    \textbf{Step 2.} Let $\gamma = (\gamma_{v'})_{v'} \in \bigoplus_{v' \in V'} \CH_1(Y'_{v'},\Lambda)$ be a one-cycle satisfying \eqref{eq:gamma-assumption} and let $q \colon Y' \to Y$ be the natural morphism.  We will show that for every $e \in E$
    \begin{equation}\label{eq:Phi-Psi}
        \sum\limits_{n = 1}^{r-1} n \cdot \Phi^{\Lambda}_{Y',Y'_{(e,n)}}(\gamma) = \Psi^\Lambda_{Y,Y_e}(q_\ast \gamma) \in \CH_0(Y_e,\Lambda),
    \end{equation}
    where we view the zero-cycles on the left hand side as zero-cycles on $Y_e$ using \eqref{eq:projective-bundle-CH0}.

    To prove the above claim, first note that by assumption there exists $\alpha_{(e,n)} \in \CH_0(Y_{e},\Lambda)$ for every $(e,n) \in E \times \{1,\dots,r-1\}$ such that $$
        \gamma_{(e,n)} = q^\ast_{(e,n)} \alpha_{(e,n)} \in \CH_1(Y'_{(e,n)},\Lambda) .
    $$
    To simplify the formulas below, we set additionally 
    \begin{align*}
        \alpha_{(e,0)} &\coloneq  \gamma_{v(e)}|_{Y_{e}} \in \CH_0(Y_{e},\Lambda), \\
        \alpha_{(e,r)} &\coloneq  \gamma_{w(e)}|_{Y_{e}} \in \CH_0(Y_{e},\Lambda),
    \end{align*}
    where $\gamma_{v(e)}|_{Y_{e}}$ is the pullback of $\gamma_{v(e)}$ along the regular embedding $Y_e \hookrightarrow Y_{v(e)}$ of the Cartier divisor $Y_e \subset Y_{v(e)}$.
    Note that we used here the isomorphisms $Y'_{v(e)} \cong Y_{v(e)}$. By \eqref{eq:Phi}, we see that for such a collection of one-cycles $\gamma=(\gamma_{v'})_{v'}$, for $e \in E$ and $1 \leq n < r$, 
    \begin{align*}
        \Phi^{\Lambda}_{Y',Y'_{(e,n)}}(\gamma) &= - 2 \alpha_{(e,n)} + \alpha_{(e,n-1)} + \alpha_{(e,n+1)} \in \CH_0(Y_{e},\Lambda) \cong \CH_0(Y'_{(e,n)},\Lambda).
    \end{align*}
    The following computation then shows the claim in Step 2
    $$
    \begin{aligned}
         \sum\limits_{n = 1}^{r-1} n \cdot \Phi^{\Lambda}_{Y',Y'_{(e,n)}}(\gamma) &= \sum\limits_{n = 1}^{r-1} n \left(- 2\alpha_{(e,n)} + \alpha_{(e,n-1)} + \alpha_{(e,n+1)}\right) \\
        &= \alpha_{(e,0)} - r \alpha_{(e,r-1)} + (r-1) \alpha_{(e,r)} + \sum\limits_{n=1}^{r-2} (-2n + n+1 + n -1) \alpha_{(e,n)} \\
        &= \alpha_{(e,0)} - \alpha_{(e,r)} \\
        &= \gamma_{v(e)}|_{Y_{e}} - \gamma_{w(e)}|_{Y_{e}} \\
        &= \Psi^\Lambda_{Y,Y_e}(q_\ast \gamma),
    \end{aligned}
    $$
    where we used in the third equality that $r$ is divisible by the characteristic $c$ of $\Lambda$.

    We finish the proof of the proposition.
    To this end, let $$
        z = (z_{e})_{e} \in \bigoplus\limits_{e \in E} \CH_0(Y_{e},\Lambda)
    $$
    be a collection of zero-cycles. By assumption, there exists a one-cycle $\gamma \in \bigoplus_{v' \in V'} \CH_1(Y'_{v'},\Lambda)$ such that $\Phi^\Lambda_{Y'}(\gamma) = \beta$, where $\beta$ is the collection of zero-cycles $(\beta_{v'})_{v'}$ in $\bigoplus_{v' \in V'} \CH_0(Y'_{v'},\Lambda)$ with
    $$
        \beta_{v'} = \begin{cases}
            m \cdot z_{e} & \text{for } v' = (e,1), \\
            0 & \text{otherwise.}
        \end{cases}
    $$
    By Step 1, we can assume that $\gamma$ satisfies the condition \eqref{eq:gamma-assumption}. Then Step 2 shows that $m \cdot z = \Psi^\Lambda_Y(q_\ast \gamma)$ is contained in the image of $\Psi^\Lambda_{Y}$, as we want.
\end{proof}

\begin{remark}
    The key point in the above proof is \eqref{eq:Phi-Psi}.
    In the chain of equalities showing \eqref{eq:Phi-Psi}, we used that $\Lambda$ has positive characteristic $c$ and $r$ is divisible by $c$.
We do not know how to perform this step integrally, despite the fact that the (more general) analysis of $\Phi$ under base changes carried out in \cite{Lange} does in fact work integrally.
\end{remark}

\begin{remark}
In \cite{PavicSch}, the obstruction morphism $\Phi^\Z_{Y}$ is studied for strictly semi-stable degenerations $\mathcal X\to \Spec R$ that are proper.
Under this assumption, the image of $\Phi_Y^\Z$  is contained in the kernel of the degree map.
The obstruction to the existence of a decomposition of the diagonal used in \cite{PavicSch} is the cokernel of $\Phi_Y^\Z$, viewed as a map to the subspace of degree-zero classes.
In the above discussion, properness is dropped and so we cannot talk about the degree anymore.
This is the reason why we work directly with the cokernel of $\Phi$ and $\Psi$, respectively. 
An important difference between $\Phi$ and $\Psi$ is the fact that $\Phi$ maps to the Chow groups of zero-cycles of the components, while $\Psi$ maps to $\CH_0$ of the intersections $Y_{ij}$ of two irreducible components of $Y$.
\end{remark}

\subsection{Proof of Theorem \ref{thm:obstruction-general}}

\begin{proof}[Proof of Theorem \ref{thm:obstruction-general}]
    Let $m \coloneq  \Tor^\Lambda(\bar X,\emptyset)$. 
    Then there exists a finite field extension $F/K$ such that $\Tor^\Lambda(X_F,\emptyset) = m$. 

    A suitable localization $R'$ of the integral closure of $R$ in $F$ is also a discrete valuation ring with fraction field $F$ and residue field $k$. (Note that $k$ is algebraically closed.) Up to replacing $R'$ by a ramified extension, we can assume that the ramification index $r$ of $R'/R$ is divisible by the characteristic $c$ of $\Lambda$.
    
    Recall that the special fibre of $\mathcal{X} \to \Spec R$ is an snc scheme $Y$ which has no triple intersections by assumption and we denote its dual graph by $(V,E)$.
    By Lemma \ref{lem:basechange+resolution-general}, there exists a resolution $\tilde {\mathcal{X}} \to \mathcal{X} \times_R R'$ by repeatedly blowing-up the non-Cartier components of the special fibre. 
    The generic fibre of $\tilde {\mathcal{X}} \to \Spec R'$ is isomorphic to $X_F$ and the special fibre $Y'$ has no triple intersections and its dual graph $(V',E')$ is as in Lemma \ref{lem:basechange+resolution-general} \eqref{item:prop:basechange+res-general:dualgraph}. 
    It then suffices by Proposition \ref{prop:basechange+resolution-general}  to show that $\coker \Phi^\Lambda_{Y'_L}$ is $m$-torsion for all field extensions $L/k$.

    Let $L/k$ be a field extension. By inflation of local rings (see e.g. \cite[Chapter IX, Appendice \S 2, Corollaire du Th\'eor\`eme 1 and Exercice 4]{bourbaki}), there exists an unramified extension of discrete valuation rings $A/R'$ such that the induced extension of residue fields is $L/k$.
    Passing to the completion, we may in addition assume that $A$ is complete. 
    We consider the base-change 
    $$
    \tilde {\mathcal{X}}_A\coloneq  \tilde {\mathcal{X}} \times_{R'} A \to \Spec A,
    $$
    which is a strictly semi-stable $A$-scheme, see e.g.\ \cite[Proposition 1.3]{hartl}. We aim to show that the cokernel of the map
    $$
        \Phi^\Lambda_{Y'_L} \colon \bigoplus\limits_{v' \in V'} \CH_1(Y'_{v',L},\Lambda) \longrightarrow \bigoplus\limits_{w' \in V'} \CH_0(Y'_{w',L},\Lambda)
    $$
    defined in \eqref{eq:Phi} is $m$-torsion, where $Y'_{v',L} \coloneq  Y'_{v'} \times_k L$. Let $$
        z = (z_{w'})_{w'} \in \bigoplus\limits_{w' \in V'} \CH_0(Y'_{w',L},\Lambda)
    $$    
    be a collection of zero-cycles. (By a moving lemma, we can assume that no $z_{w'}$ lies in the intersection of $Y'_{w'}$ with another component $Y'_{v'}$.) 
    By Hensel's lemma, see \cite[Th\'eor\`eme 18.5.17]{EGA4.4}, there exists a horizontal one-cycle $h \in Z_1(\tilde {\mathcal{X}}_A,\Lambda)$ such that
    \begin{equation}\label{eq:hrestricted}
        \left.h\right|_{Y'_{w',L}} = z_{w'} \in \CH_0(Y'_{w',L},\Lambda)
    \end{equation}
    holds for all $w'$ already on the level of cycles (without rational equivalence). 
    Recall that $m$ is the $\Lambda$-torsion order of $X_F$ with respect to the empty set.
    Lemma \ref{lem:weak-decomposiiton-of-diagonal} together with Nagata's compactification theorem thus implies that $\CH_0(X_F \times_F F',\Lambda)$ is $m$-torsion for all field extension $F'/F$.
    It follows that the restriction of $m \cdot h$ to the generic fibre $\tilde {\mathcal{X}}_A \times_A F' = X_F \times_F F'$ vanishes, where $F'$ is the fraction field of $A$. Thus the horizontal one-cycle $m \cdot h$ is rationally equivalent to a cycle $\gamma$ supported on the special fibre $Y'_L$ by the localization exact sequence, see \cite[\S 1.8]{fulton}. 
    Hence, we see from \eqref{eq:hrestricted} and Lemma \ref{lem:alternativePhi} that $m \cdot z = \Phi^\Lambda_{Y'_L}(\gamma)$ is contained in the image of $\Phi^\Lambda_{Y'_L}$. 
    This shows that the cokernel of the map $\Phi^\Lambda_{Y'_L}$ is $m$-torsion, which finishes the proof of the theorem.
\end{proof}

\begin{remark}\label{rem:triple-intersection}
    The assumption that the special fibre has no triple intersections in Theorem \ref{thm:obstruction-general} appears to be crucial. 
    In particular, there seems to be no useful generalization of $\Psi$ to snc schemes with deeper strata,
    despite the fact that the definition of $\Phi$ in \eqref{eq:Phi} makes sense in more generality, see \cite[Section 3]{PavicSch}.
    Indeed, if $Z=Y_{i_1}\cap \dots \cap Y_{i_n}$ is a strata of the special fibre $Y$, then we can always perform an $n:1$ base change followed by a resolution as in \cite{hartl} to arrive at a special fibre $\tilde Y$ that contains a component $P_Z$ that is birational to a projective bundle over $Z$.
    However,  for $n\geq 3$, one can show via similar arguments as in \cite{Lange} that the diagonal point of $P_Z$ will automatically be in the image of $\Phi_{\tilde Y_{k(P_Z)}}$.
    (The key difference to $n\leq 2$ is that the blow-up of a component $Y_{i_j}$ along $Z$ contains a positive dimensional projective bundle over $Z$ if $n\geq 3$.) 
    Hence, the strategy for disproving (retract) rationality in \cite{PavicSch,Lange} cannot be applied to strata given by the intersection of more than 2 components. 
%    
%    The observation that our cycle-theoretic analog of the motivic obstruction to stable rationality from \cite{NS,KT} is not sensitive to obstructions that lie in strata of high codimension may hint at the difference between stable rationality and retract rationality or $\A^1$-connectedness, which is known over certain non-closed fields \cite{endo-miyata} but open over closed fields.
\end{remark}

\section{Double cone construction}\label{section:doublecone}

In this section we consider an explicit degeneration of a variety birational to a degree $d$ hypersurface $X'$ into a union of two rational varieties whose intersection $Z$ is a degree $d$ hypersurface of lower dimension. 
We aim to apply Theorem \ref{thm:obstruction-proper} to this particular family and show that the $\Lambda$-torsion order of $Z$ divides the $\Lambda$-torsion order of $X'$. In particular, we can inductively increase the dimension of retract irrational hypersurfaces.

Let $k = \overline{k_0(\lambda)}$ be an algebraic closure of the purely transcendental field extension $k_0(\lambda)$ of an algebraically closed field $k_0$. Write $N = n + r+ s$ for some integers $n,r,s \in \Z_{\geq 0}$ and consider integers $d \geq 4$ and $l \geq 1$ such that $2l \leq d$. We denote the homogeneous coordinates of $\CP_k^{N+3}$ by $x_0,\dots,x_n,y_1,\dots,y_{r+1},z_1,\dots,z_s,z,w$. Let \begin{equation}\label{eq:polynomials}
    f, a_0, a_1, \dots, a_l \in k_0[x_0,\dots,x_n,y_1,\dots,\hat{y_j},\dots,y_{r+1},z_1,\dots,z_s]
\end{equation}
be homogeneous polynomials of degree $\deg f = d$ and $\deg a_i = d-2i$ which do  not contain the variable $y_j$ for some $1 \leq j \leq r+1$.
Assume that \begin{equation}\label{eq:conditionirreducible}
    f + a_0 \in k[x_0,\dots,x_n,y_1,\dots,\hat{y_j},\dots,y_{r+1},z_1,\dots,z_s] \text{ is irreducible.}
\end{equation} 
Let $R \coloneq  k[t]_{(t)}$ be the local ring of $\A^1$ at the origin and consider the complete intersection $R$-scheme
\begin{equation}\label{eq:mathcalX}
    \mathcal{X} \coloneq   \left\{f + \sum\limits_{i=0}^l a_ix_0^iy_j^i + x_0^{d-1}z + x_0^{d-2}(\lambda y_j + x_0)w = tx_0^2 + zw = 0\right\} \subset \CP_R^{N+3}.
\end{equation} 

\begin{lemma}\label{lem:singmathcalX}
    The singular locus of $\mathcal X$ in \eqref{eq:mathcalX} is contained in $\{x_0 = 0\} \subset \mathcal X$.
\end{lemma}

\begin{proof}
    Consider the part of the Jacobian given by the derivatives $\partial_t$ and $\partial_z$
    $$
        \begin{pmatrix}
            0 & x_0^{d-1} \\
            x_0^2 & w
        \end{pmatrix}.
    $$
    As the singular locus is given by the vanishing of all $2 \times 2$ minors of the Jacobian, we see that it is contained in $\{x_0 = 0\}$ as claimed.
\end{proof}

The generic fibre of the family $\mathcal X \to \Spec R$ in \eqref{eq:mathcalX} is birational to a degree $d$ hypersurface. 
For the inductive argument to work, it will be important to understand the corresponding birational map, which is the content of the following lemma, where we denote by $K\coloneq k(t)$ the fraction field of $R=k[t]_{(t)}$.

\begin{lemma}\label{lem:genericfibre}
The generic fibre $X = \mathcal X \times_R K$ of the family \eqref{eq:mathcalX} is birational to a geometrically integral degree $d$ hypersurface $X'$ of the form
    \begin{equation}\label{eq:X'}
        X' = \left\{f + \sum\limits_{i=0}^l a_i' y_j^i = 0\right\} \subset \CP^{N+2}_K,
    \end{equation}
    where $f$ is as in \eqref{eq:polynomials} and $a'_0,\dots,a'_l \in K[x_0,\dots,x_n,y_1,\dots,\hat{y_j},\dots,y_{r+1},z_1,\dots,z_{s+1}]$ are homogeneous polynomials of degree $\deg a'_i = d - i$. (An explicit formula for them is given in \eqref{eq:aiprime} below.) Moreover, they satisfy the following properties:
    \begin{enumerate}[(1)]
        \item The birational map induces an isomorphism between the open subsets $\{x_0 \neq 0\} \subset X$ and $\{x_0z_{s+1} \neq 0\} \subset X'$; \label{item:lem:genericfibre1}
        \item If there exists $e \in \N$ such that $x_0^{ie} \mid a_i$ for all $i = 0,\dots,l$, then $x_0^{ie} \mid a'_i$ for all $i = 0,\dots,l$; \label{item:lem:genericfibre2}
        \item If all $a_0,\dots,a_l$ do not contain one of the coordinates $x_1,\dots,x_n$, $y_1,\dots,y_{r+1},$ or $z_1,\dots,z_{s}$, then so do all $a'_0,\dots,a'_l$; \label{item:lem:genericfibre3}
        \item The polynomial $f + a_0'$ is irreducible over $\Bar{K}$. \label{item:lem:genericfibre4} 
    \end{enumerate}
\end{lemma}

\begin{proof}
The parameter $t$ is nonzero on the generic fibre $X$
of \eqref{eq:mathcalX}.
We then work on the open subset $\{x_0z\neq 0\}$  and perform the substitution $w=-tx_0^2z^{-1}$ to obtain the equation
\begin{align} \label{eq:X-double-cone}
        f + \sum\limits_{i=0}^l a_i x_0^i y_j^i + x_0^{d-1}z - x_0^{d-2}(\lambda y_j + x_0)tx_0^2 z^{-1} = 0.
\end{align}
    After the change of coordinates $y_j \mapsto y_jz / x_0 - \lambda^{-1} x_0$ we arrive at the degree $d$ hypersurface given by the vanishing of the polynomial
    $$
        f + \sum\limits_{i = 0}^l a_i (y_j z - \lambda^{-1}x_0^2)^i + x_0^{d-1}z - t\lambda x_0^{d-1}y_j = 0.
    $$
    Reordering the terms and renaming the coordinate $z$ to $z_{s+1}$, yields the claim that $X$ is birational to a degree $d$ hypersurface $X'$ of the form \eqref{eq:X'}, where
    \begin{equation}\label{eq:aiprime}
        a'_i = z_{s+1}^i \left(\sum\limits_{m = i}^l \binom{m}{i} (-\lambda^{-1})^{m-i} x_0^{2 m - 2 i} a_{m}\right) - \delta_{i,1} t \lambda x_0^{d-1} + \delta_{i,0} x_0^{d-1}z_{s+1} .
    \end{equation}
    This shows \eqref{item:lem:genericfibre2} and \eqref{item:lem:genericfibre3}. 
    Item \eqref{item:lem:genericfibre1} follows immediately from the above construction. 
    Indeed, the coordinate transformation $y_j \mapsto y_jz/x_0 - \lambda^{-1}x_0$ is invertible on the set $\{x_0z \neq 0\}$ with inverse given by $y_j \mapsto x_0(y_j + \lambda^{-1}x_0)/z$. Thus the birational map induces an isomorphism between the open subsets $\{x_0 \neq 0\} = \{x_0z \neq 0\} \subset X$ and $\{x_0 z_{s+1} \neq 0\} \subset X'$. (Note that we renamed the $z$-coordinate to $z_{s+1}$.) 
    Next we prove \eqref{item:lem:genericfibre4}. 
    Since $a_0'$ contains the variable $z_{s+1}$ linearly by \eqref{eq:aiprime} and $f$ does not contain $z_{s+1}$, the condition \eqref{eq:conditionirreducible} implies that $f + a_0'$ is irreducible, as claimed. 
    The hypersurface $X'$ is geometrically integral by \eqref{item:lem:genericfibre4}.
\end{proof}

\begin{remark}
The name ``double cone construction'' is taken from \cite{Moe}; it is reflected by the fact that the generic fibre of our degeneration can birationally be described by the equation \eqref{eq:X-double-cone} above, which contains the variable $z$ and its inverse linearly and so its Newton polytope is a double cone.
We note however that the degenerations that we use in this paper in general do not have toric singularities (e.g.\ because the singular hypersurfaces in \cite{Sch-JAMS} do not have toric singularities) and hence is different from the degenerations suitable for the method of \cite{NO,Moe}.
\end{remark}

\begin{corollary}\label{cor:geomintegral}
    The generic fibre $X$ of \eqref{eq:mathcalX} is geometrically integral.
\end{corollary}

\begin{proof}
    Since $X$ is a complete intersection in $\CP^{N+3}_K$, it is equidimensional and Cohen-Macaulay. 
    As $X$ is Cohen-Macaulay, $X$ has no embedded components. 
    By Lemma \ref{lem:genericfibre}, we know that the open subset $\{x_0 \neq 0\} \subset X$ is isomorphic to an open subset of the geometrically integral hypersurface $X' \subset \CP^{N+2}_K$ in \eqref{eq:X'}. Hence it suffices to show that the subset $\{x_0 = 0\} \subset X$ is not an irreducible component of $X$, which is clear because $f + a_0$ is an irreducible polynomial by \eqref{eq:conditionirreducible}.
\end{proof}

We turn to the special fibre of the family \eqref{eq:mathcalX}. The special fibre $Y = \mathcal X \times_R k$ of the family \eqref{eq:mathcalX} has two components, namely
\begin{align}\label{eq:Y0}
    Y_0 &\coloneq  \left\{f + \sum\limits_{i=0}^l a_i x_0^i y_j^i + x_0^{d-1} z = 0\right\} \subset \CP_k^{N+2} ,\\ 
\label{eq:Y1}
    Y_1 &\coloneq  \left\{f + \sum\limits_{i=0}^l a_i x_0^i y_j^i + x_0^{d-2}(\lambda y_j + x_0) w = 0 \right\} \subset \CP_k^{N+2}.
\end{align}
The intersection $Z \coloneq  Y_0 \cap Y_1$ is the degree $d$ hypersurface
\begin{equation}\label{eq:Z}
    Z \coloneq  \left\{f + \sum\limits_{i=0}^l a_i x_0^i y_j^i = 0\right\} \subset \CP_k^{N+1},
\end{equation}
The assumption (see \eqref{eq:conditionirreducible}) that $f + a_0$ is irreducible implies that $Z$ is integral.

\begin{lemma}\label{lem:singspecialfibre}
    The singular locus of $Y_0$ is contained in $\{x_0 = 0\}$ and the singular locus of $Y_1$ is contained in the closed subset $\{x_0(\lambda y_j + x_0) = 0\}$. 
    Moreover, $$
        Y_1^{\sing} \cap Z \subset Z^{\sing}.
    $$
\end{lemma}

\begin{proof}
    The derivative of the defining equation of $Y_0$ with respect to $z$ is given by
    $$
        \partial_z \left(f + \sum\limits_{i = 0}^l a_ix_0^iy_j^i + x_0^{d-1}z\right) = x_0^{d-1}.
    $$
    Hence, the singular locus of $Y_0$ is contained in $\{x_0 = 0\}$.

    The derivative of the defining equation of $Y_1$ with respect to $w$ is given by
    $$
        \partial_w \left(f + \sum\limits_{i = 0}^l a_ix_0^iy_j^i + x_0^{d-2}(\lambda y_j + x_0)w\right) = x_0^{d-2} (\lambda y_j + x_0).
    $$
    Hence, the singular locus of $Y_1$ is contained in $\{x_0(\lambda y_j + x_0)=0\}$ as claimed. 
    Finally, the last claim in the lemma is clear, as $Z \subset Y_1$ is a Cartier divisor in $Y_1$ (given by $\{w = 0\} \subset Y_1$).
\end{proof}

    In order to understand the obstruction map \eqref{eq:Psi} for the family \eqref{eq:mathcalX}, we need to control the Chow group of the two components $Y_0$ and $Y_1$.
    
    \begin{lemma}\label{lem:CH1}
        Let $Y_0$ and $Y_1$ be as in \eqref{eq:Y0} and \eqref{eq:Y1}, i.e.\ the irreducible components of the special fibre of the family \eqref{eq:mathcalX}.
        Consider the divisors $D_0 \coloneq  \{x_0 = 0\} \subset Y_0$ and $D_1 \coloneq  \{x_0(\lambda y_j + x_0) = 0\} \subset Y_1$. 
        Then the natural push-forward maps
        $$
            \CH_1(D_0 \times_k L) \longrightarrow \CH_1(Y_0 \times_k L), \quad \CH_1(D_1 \times_k L) \longrightarrow \CH_1(Y_1 \times_k L)
        $$
        are surjective for every field extension $L/k$.
    \end{lemma}
    
    \begin{proof}
     Let $L/k$ be a field extension. Recall that \begin{align*}
        Y_0 \times_k L &= \left\{f + \sum\limits_{i=0}^l a_i x_0^i y_j^i + x_0^{d-1}z = 0\right\} \subset \CP_L^{N+2}, \\
        Y_1 \times_k L &= \left\{f + \sum\limits_{i=0}^l a_i x_0^i y_j^i + x_0^{d-2}(\lambda y_j + x_0)w = 0\right\}  \subset \CP_L^{N+2}.
    \end{align*}
    We consider $Y_1 \times_k L$. The projection away from $P = [0:\dots:0:1] \in \CP_L^{N+2}$ induces a rational map $$
        \varphi \colon \CP_L^{N+2} \dashrightarrow \CP_L^{N+1}.
    $$ 
    Since $w$ appears only linearly in the defining equation of $Y_1 \times_k L$, the restriction of $\varphi$ to $Y_1 \times_k L$ yields a birational map
    $$
        Y_1 \times_k L \dashrightarrow \CP_L^{N+1},
    $$
    which induces an isomorphism between the complements of the closed subschemes $D_1 \times_k L \subset Y_1 \times_k L$ and $H_1 \coloneq  \{x_0 (\lambda y_j + x_0)=0\} \subset \CP_L^{N+1}$, respectively. 
    Since $H_1$ is a union of two hyperplanes in $\CP_L^{N+1}$, the pushforward along the natural inclusion
    $$
        \CH_1(H_1) \longrightarrow \CH_1(\CP^{N+1}_L) \cong \Z
    $$
    is surjective. 
    Thus, by the localization exact sequence (see \cite[Proposition 1.8]{fulton}), we find that $$
        \CH_1(D_1 \times_k L) \longrightarrow \CH_1(Y_1 \times_k L)
    $$
    is surjective, as $\CH_1(Y_1 \times_k L \setminus D_1 \times_k L) \cong \CH_1(\CP^{N+1}_L \setminus H_1) = 0$ by the above discussion. A similar argument shows that $$
        \CH_1(D_0 \times_k L) \longrightarrow \CH_1(Y_0 \times_k L)
    $$
    is surjective. This finishes the proof of the lemma.
\end{proof}

The following proposition is the main result of this section; it will be used for the induction step in our inductive argument.

\begin{proposition}\label{prop:inductionstep}
    Let $\Lambda$ be a ring of positive characteristic such that the exponential characteristic of $k_0$ is invertible in $\Lambda$. 
    Let $\mathcal X \to \Spec R$ be the projective family from \eqref{eq:mathcalX} with generic fibre $X$ for some $l \geq 2$. Let $Y_0$ and $Y_1$ be the irreducible components of the special fibre as in \eqref{eq:Y0} and \eqref{eq:Y1} and denote their scheme-theoretic intersection by $Z \coloneq  Y_0 \cap Y_1$. 
    Let $$
        h \in R[x_0,\dots,x_n,y_1,\dots,y_{r+1},z_1,\dots,z_s]
    $$
    be any homogeneous polynomial such that its reduction $h_0$ modulo the maximal ideal in $R$ has coefficients in $k_0$, i.e.\ $h_0 \in k_0[x_0,\dots,x_n,y_1,\dots,y_{r+1},z_1,\dots,z_s]$  and such that $X$ is smooth over $K=k(t)$ away from $W_X \coloneq  \{x_0 h = 0\} \subset X$.
    Then
    $$
        \Tor^\Lambda(Z,W_Z) \mid \Tor^\Lambda(\bar{X},W_{\bar{X}}),
    $$
    where $W_Z \coloneq  \{x_0 h_0 = 0\} \cup Z^{\sing} \subset Z$ and  $W_{\bar{X}} \coloneq  W_X \times_K \bar K  \subset \bar X \coloneq  X\times_K\bar K$. 
    In particular, $$
        \Tor^\Lambda(Z,W_Z) \mid \Tor^\Lambda(\bar{X}',W'),
    $$
    where $X'$ is as in Lemma \ref{lem:genericfibre}, $\bar{X}' \coloneq  X' \times_K \bar{K}$, and $W' \coloneq  \{x_0 z_{s+1} h = 0\} \subset \bar{X}'$.
\end{proposition}

\begin{remark}
As a consequence of Lemma \ref{lem:singmathcalX}, the smoothness of $X \setminus W_X$ is automatically satisfied in characteristic $0$.
This is not true in general, but we will choose in our applications $W_X$ carefully so that the condition holds (over any field). 
\end{remark} 

\begin{proof}
    Since $\Lambda$ is a ring of positive characteristic, 
    $$
    m\coloneq \Tor^\Lambda(\bar{X},W_{\bar{X}})
    $$ is a positive integer $m \in \Z_{\geq 1}$. 

    \textbf{Step 1.} 
    We will check that the assumptions in Theorem \ref{thm:obstruction-proper} are satisfied for the projective family $\mathcal{X} \to \Spec R$
    from \eqref{eq:mathcalX} with the closed subset $$
        W_{\mathcal X} \coloneq  \{x_0h = 0\} \cup Y_1^{\sing} \cup Z^{\sing} \subset \mathcal X.
    $$
    As in Theorem \ref{thm:obstruction-proper}, let $W_Y \coloneq  W_{\mathcal{X}} \cap Y$. We note that $W_X$ agrees with $W_{\mathcal{X}} \cap X$.
   The generic fibre $X$ of $\mathcal X \to \Spec R$ is geometrically integral by Corollary \ref{cor:geomintegral}. 
  The special fibre $Y^\circ \coloneq  Y \setminus W_Y$ consists of two components $Y^\circ = Y_0^\circ \cup Y_1^\circ$ such that $Y_0^\circ$, $Y_1^\circ$, and their intersection $Z^\circ \coloneq  Y_0^\circ \cap Y_1^\circ$ are smooth and integral, see also Lemma \ref{lem:singspecialfibre}.
  In particular, $Y^\circ$ is an snc scheme, see Definition \ref{def:snc}.
    The singular locus of $\mathcal X$ is contained in $\{x_0 = 0\}$ by Lemma \ref{lem:singmathcalX}. 
    It follows from this that $Y^\circ_i$ is a Cartier divisor on $\mathcal X^\circ$ for $i=0,1$. The generic fibre of the $R$-scheme $\mathcal{X}^\circ \coloneq  \mathcal X \setminus W_{\mathcal X}$ is equal to $X \setminus W_X$ and thus smooth by assumption. 
    In particular, $\mathcal{X}^\circ\to \Spec R$ is strictly semi-stable, see Definition \ref{def:strictly-semi-stable}.
    It follows that the assumptions \eqref{item:thm:obstruction-proper-strictlysemistable} and \eqref{item:thm:obstruction-proper-3intersection} in Theorem \ref{thm:obstruction-proper} are satisfied for $\mathcal X \to \Spec R$ and the closed subset $W_{\mathcal{X}} \subset \mathcal X$.
    This concludes Step 1.

Recall that $Z$ is integral and let $\delta_{Z^\circ}\in  Z^\circ_{k(Z)}$ denote the diagonal point of $Z^\circ$, which is dense open in $Z$.

\textbf{Step 2.}
We will show that there is a one-cycle $\gamma \in \CH_1(Y_{1} \times_k k(Z),\Lambda)$, supported on $\{x_0(\lambda y_j + x_0) =0 \} \subset Y_1\times_k k(Z)$, such that
    \begin{equation}\label{eq:deltaZ}
        m \cdot \delta_{Z^\circ} = (\iota^\ast \gamma)|_{Z^\circ_{k(Z)}} \in \CH_0(Z^\circ_{k(Z)},\Lambda) ,
    \end{equation}
    where $ \iota \colon Y_1^\circ \times_k k(Z) \hookrightarrow Y_1 \times_k k(Z)$ denotes the natural open embedding and $(\iota^\ast \gamma)|_{Z^\circ_{k(Z)}}$ the pullback of the one-cycle to the Cartier divisor $Z^\circ_{k(Z)} \subset Y_1^\circ \times_k k(Z)$ along the natural regular embedding.

    By Step 1, Theorem \ref{thm:obstruction-proper} implies that the cokernel of the map $$
        \Psi^\Lambda_{Y_L^\circ} \colon \CH_1(Y_0^\circ \times_k L,\Lambda) \oplus \CH_1(Y_1^\circ \times_k L,\Lambda) \longrightarrow \CH_0(Z^\circ \times_k L,\Lambda)
    $$
    from \eqref{eq:Psi} (with $0 < 1$) is $m$-torsion for every field extension $L/k$.
    In particular,  
    $$
        m \cdot \delta_{Z^\circ} \in \im \Psi^\Lambda_{Y^\circ_{k(Z)}}.
    $$
     Note that $\CH_1(Y_0^\circ \times_k k(Z),\Lambda) = 0$ by Lemma \ref{lem:CH1} and that the pull-back
    $$
        \iota^\ast \colon \CH_1(Y_1 \times_k k(Z),\Lambda) \longrightarrow \CH_1(Y_1^\circ \times_k k(Z),\Lambda)
    $$ 
    is surjective by \cite[Proposition 1.8]{fulton}. 
    Hence there exists a one-cycle $\gamma \in \CH_1(Y_{1} \times_k k(Z),\Lambda)$ such that $ m \cdot \delta_{Z^\circ} = (\iota^\ast \gamma)|_{Z^\circ_{k(Z)}}$. 
    By Lemma \ref{lem:CH1} we can further assume that $\gamma$ is supported on $\{x_0(\lambda y_j + x_0) =0 \} \subset Y_1\times_k k(Z)$.
    This concludes Step 2.

\textbf{Step 3.}
    We specialize now $\lambda \to 0$ and aim to compute the image of \eqref{eq:deltaZ} under the corresponding specialization map on Chow groups (see Lemma \ref{lem:specialization-functoriality}); we will show that the specialization of the one-cycle $\gamma$ vanishes and so does the specialization of $m\cdot \delta_{Z^\circ}$.  
    
    Consider the discrete valuation ring $B = k_0[\lambda]_{(\lambda)}$ with residue field $k_0$ and fraction field $k_0(\lambda)$. Recall that $k$ is an algebraic closure of $k_0(\lambda)$. Then consider the flat projective $B$-schemes
    $$
        \begin{aligned}
            \mathcal{Z} &\coloneq  \left\{f + \sum\limits_{i=0}^l a_i x_0^i y_j^i = 0\right\} \subset \CP^{N+1}_B, \\
            \mathcal{Y}_1 &\coloneq  \left\{f + \sum\limits_{i=0}^l a_i x_0^i y_j^i + x_0^{d-2}(\lambda y_j + x_0)w = 0\right\} \subset \CP^{N+2}_B,
        \end{aligned}
    $$
    where $f$ and $a_i$ are as in \eqref{eq:polynomials}. 
    Note that $Z$ and $Y_1$ are the geometric generic fibre of $\mathcal{Z}$ and $\mathcal{Y}_1$, respectively.
    Let $\mathcal{Z}^\circ \subset \mathcal{Z}$ and $\mathcal{Y}_{1}^\circ \subset \mathcal{Y}_1$ be the complement of the closure of $W_Z$ in $\mathcal{Z}$ and of $W_{Y_1}$ in $\mathcal{Y}_1$, respectively.
    Note that $Z^\circ = Z \setminus W_Z$ and $Y_1^\circ = Y_1 \setminus W_{Y_1}$ are the geometric generic fibres of $\mathcal{Z}^\circ$ and $\mathcal{Y}_1^\circ$, respectively. We denote the special fibres by $Z_0^\circ$ and $Y_{1,0}^\circ$, respectively. 
    By Lemma \ref{lem:specialization-functoriality}, there exist specialization maps induced by Fulton's specialization map for the flat $\mathcal{O}_{\mathcal{Z},Z_0}$-schemes $\mathcal{Z}^\circ \times_B \mathcal{O}_{\mathcal{Z},Z_0}$, $\mathcal{Y}_1^\circ \times_B \mathcal{O}_{\mathcal{Z},Z_0}$, and $\mathcal{Y}_1 \times_B \mathcal{O}_{\mathcal{Z},Z_0}$    
    $$
        \begin{aligned}
            \specialization_{Z^\circ} &\colon \CH_0(Z^\circ \times_k k(Z),\Lambda) \longrightarrow \CH_0(Z_0^\circ \times_{k_0} k_0(Z_0),\Lambda), \\
            \specialization_{Y_1^\circ} &\colon \CH_1(Y_1^\circ \times_k k(Z),\Lambda) \longrightarrow \CH_1(Y_{1,0}^\circ \times_{k_0} k_0(Z_0),\Lambda), \\
            \specialization_{Y_1} &\colon \CH_1(Y_1 \times_k k(Z),\Lambda) \longrightarrow \CH_1(Y_{1,0} \times_{k_0} k_0(Z_0),\Lambda),
        \end{aligned}
    $$
    where $Y_{1,0}$ denotes the special fibre of the $B$-scheme $\mathcal{Y}_1$. 

    We apply now the specialization $\specialization_{Z^\circ}$ to the zero-cycle \eqref{eq:deltaZ}. By Lemma \ref{lem:specialization-functoriality}, we get
    \begin{equation}\label{eq:SpecializationdeltaZ}
        m \cdot \delta_{Z_0^\circ} = \specialization_{Z^\circ}(m \cdot \delta_{Z^\circ}) = \specialization_{Z^\circ}((\iota^\ast \gamma)|_{Z^\circ_{k(Z)}}) = \specialization_{Y_1^\circ}(\iota^\ast \gamma)|_{Z^\circ_{0,k_0(Z_0)}},
    \end{equation}
    where $|_{Z^\circ_{k(Z)}}$ and $|_{Z^\circ_{0,k_0(Z_0)}}$ denote the pullback along the regular embedding of the Cartier divisor $Z^\circ_{k(Z)}$ and $Z^\circ_{0,k_0(Z_0)}$ in $Y_1^\circ \times_k k(Z)$ and $Y_{1,0}^\circ \times_{k_0} k_0(Z_0)$, respectively. 
    Recall that the one-cycle $\gamma$ is supported on $\{x_0(\lambda y_j + x_0) =0 \} \subset Y_1\times_k k(Z)$. Thus the specialization $\specialization_{Y_1}(\gamma) \in \CH_1(Y_{1,0},\Lambda)$ is supported on $\{x_0^2 =0\} \subset Y_{1,0}$. In particular
    $$
        \specialization_{Y_1^\circ}(\iota^\ast \gamma) = 0 \in \CH_1(Y_{1,0}^\circ),
    $$
    as the subset $\{x_0 = 0\} \subset Y_{1,0}$ is contained in the specialization of $W_{Y_1}$ and $\specialization$ commutes with pullbacks along open immersions. 
    Hence, the right hand side of \eqref{eq:SpecializationdeltaZ} vanishes, which concludes Step 3. 
    
    By Step 3, 
    $$
        \Tor^\Lambda(Z_0,W_{Z_0}) \mid m,
    $$
    where $W_{Z_0} \subset Z_0$ is the specialization of $W_Z \subset Z$.
    We note that $Z = Z_0 \times_{k_0} k$ and $W_Z = W_{Z_0} \times_{k_0} k$, because the defining equations of $Z$ and $W_Z$ are defined over $k_0$.
    Hence, the proposition follows from Lemma \ref{lem:torsion-order-observations} \eqref{item:0:torsion-order-observations}, as $m = \Tor^\Lambda(\bar{X},W_{\bar{X}})$.
\end{proof}

\section{Base case}

Our argument will rely on an inductive application of a degeneration as in Section \ref{section:doublecone}.
For the start of the induction we will use the explicit example of a singular hypersurface with nontrivial unramified cohomology from the proof of \cite[Theorem 7.1]{Sch-torsion}. 
We recall the example in what follows.

Let $k$ be an algebraically closed field and let $m \geq 2$ be an integer coprime to the exponential characteristic of $k$. 
Let $n \geq 2$ and $r \leq 2^n-2$ be positive integers.
Let $x_0,\dots,x_n,y_1,\dots,y_{r+1}$ be the coordinates of $\CP^{n+r+1}$ and let $\pi \in k$ be an element that is transcendental over the prime field of $k$.
Consider the homogeneous polynomial from \cite[Equation (21)]{Sch-torsion}
\begin{equation}\label{eq:gfromSch-torsion}
    g(x_0,\dots,x_n) \coloneq  \pi \cdot  \left(\sum\limits_{i=0}^n x_i^{\left\lceil \frac{n+1}{m}\right\rceil}\right)^m - (-1)^n x_0^{m \left\lceil\frac{n+1}{m}\right\rceil - n} x_1 x_2 \cdots x_n
\end{equation}
in $k[x_0,\dots,x_n]$ of degree $\deg g = m \left\lceil \frac{n+1}{m}\right\rceil  \leq n+m$. 
Using this we define the homogeneous polynomial
\begin{equation*}
    F \coloneq  g(x_0,\dots,x_n) x_0^{m+n-\deg(g)} + \sum\limits_{j = 1}^{r} x_0^{n-\deg c_j} c_j(x_1,\dots,x_n) y_j^m + (-1)^n x_1 x_2 \dots x_n y_{r+1}^m
\end{equation*}
in $k[x_0,\dots,x_n,y_1,\dots,y_{r+1}]$ of degree $m+n$, where \begin{equation}\label{eq:FermatPfister}
    c_j(x_1,\dots,x_n) \coloneq  (-x_1)^{\varepsilon_1} (-x_2)^{\varepsilon_2} \cdots (-x_n)^{\varepsilon_n} \in k[x_1,\dots,x_n]
\end{equation}
with $\varepsilon_i$ the $(i-1)$-th digit in the $2$-adic representation of $j$, i.e.\ $j = \sum_{i} \varepsilon_i 2^{i-1}$ with $\varepsilon_i \in \{0,1\}$.
Consider the associated hypersurface
\begin{equation} \label{eq:Z-Sch-torsion}
Z\coloneq \{F = 0\} \subset \CP^{n+r+1}_k
\end{equation}
of degree $m+n$.

\begin{theorem} \label{thm:InductionStart}
Let $l\in k[x_0,\dots ,x_n]$ be any nontrivial homogeneous polynomial and consider 
$$
W_{Z} \coloneq  \{ y_{r+1} l = 0\} \cup Z^{\sing} \subset Z 
$$
where $Z$ is as in \eqref{eq:Z-Sch-torsion}.
Then $\Tor^{\Z/m}(Z,W_Z)=m$.
\end{theorem}

We have the following immediate corollary.

\begin{corollary}\label{cor:Induction-Start-3}
    Let $h,l',l'' \in k[x_0,\dots,x_n]$ be nontrivial homogeneous polynomials such that $h$ is irreducible of degree $m + n + \deg l'$. Let $k' = k(\rho)$ be a purely transcendental field extension of $k$. Consider the hypersurface
    $$
        Z_\rho \coloneq  \{\rho h + l' \cdot F = 0\} \subset \CP^{n+r+1}_{k'}
    $$
    of degree $m + n + \deg l'$. Then $\Tor^{\Z/m}(\bar{Z}_\rho,W_{\bar{Z}_\rho}) = m$, where $\bar{Z}_\rho = Z_\rho \times_{k'} \Bar{k'}$ and
    $$
        W_{\bar{Z}_\rho} \coloneq   \{y_{r+1}l'l'' = 0\} \cup Z_\rho^{\sing} \times_{k'} \bar{k'}  \subset \bar{Z}_\rho. 
    $$
\end{corollary}

\begin{proof}
    Consider the pair $(Z_\rho,W_{Z_\rho})$ and let $(Z_0,W_{Z_0})$ be the pair obtained by specializing $\rho \to 0$, i.e.
    $$
        \begin{aligned}
            Z_0 &= \{l' \cdot F = 0\} \subset \CP^{n+r+1}_k, \\
            W_{Z_0} &= \{y_{r+1}l'l'' =0 \} \cup Z_{0}^{\sing} \subset Z_0.%\{l' = 0 \} \cup \{l'' =0 \} \cup Z_{0}^{\sing} \subset Z_0.
        \end{aligned}
    $$
    Note that the scheme $Z_0$ is reducible, but the open subscheme $U_0 \coloneq  Z_0 \setminus W_{Z_0} \subset Z_0$ is integral as the polynomial $F$ is irreducible. Hence, the torsion order $\Tor^{\Z/m}(Z_0,W_{Z_0})$ is defined, see Remark \ref{rem:torsion-order-X-reducible}.
    We observe that $U_0 = Z \setminus W_Z$, where the pair $(Z,W_Z)$ is as in Theorem \ref{thm:InductionStart} with $l = l' \cdot l'' \in k[x_0,\dots,x_n]$. Thus, we get $\Tor^{\Z/m}(Z_0,W_{Z_0}) = m$ by Theorem \ref{thm:InductionStart}.
    Hence, $m$ divides $\Tor^{\Z/m}(\bar{Z}_\rho,W_{\bar{Z}_\rho})$ by Lemma \ref{lem:Lambda-torsion-order-specialize}. 
    Conversely, any $\Z/m$-torsion order can be at most $m$. Hence, $\Tor^{\Z/m}(\bar{Z}_\rho,W_{\bar{Z_{\rho}}}) =m$, as claimed.
\end{proof}

\begin{proof}[Proof of Theorem \ref{thm:InductionStart}]
This follows from arguments similar to those in the proof of \cite[Theorem 6.1 and 7.1]{Sch-torsion};
we give some details for the reader's convenience.
Let $P=\{x_0=\dots=x_n=0 \}\subset \CP^{n+r+1}_k$ and consider the blow-up $Y\coloneq Bl_PZ$, which can be described via the vanishing locus of 
 \begin{equation}\label{eq:induction-start-first-resolution}  
g(x_0,\dots,x_n) x_0^{m+n-\deg(g)} y_0^m + \sum\limits_{j = 1}^{r} x_0^{n-\deg c_j} c_j(x_1,\dots,x_n) y_j^m + (-1)^n x_1 x_2\dots x_n y_{r+1}^m, 
\end{equation}  
inside the projective bundle $ \CP_{\CP^n}(\mathcal O(-1)\oplus \mathcal O^{\oplus (r+1)})$ over $\CP^n$,
where $y_0$ denotes a local coordinate that trivializes $\mathcal O(-1)$ and $y_1,\dots ,y_{r+1}$ trivialize $\mathcal O^{\oplus (r+1)}$.
The projection to the $x$-coordinates induces a morphism $f:Y\to \CP^n_k$.
We furthermore pick an alteration $\tau':Y'\to Y$ of order coprime to $m$, which can be done by \cite{temkin}, see also \cite{IT}.
The corresponding alteration of $Z$ is denoted by $\tau:Y'\to Z$.

As detailed in \cite[\S 7]{Sch-torsion}, an application of \cite[Theorem 5.3]{Sch-torsion} shows that the pullback of the class $\alpha=(x_1/x_0,\dots ,x_n/x_0)\in H^n(k(\CP^n),\mu_m^{\otimes n})$ yields an unramified class
$$
\gamma\coloneq f^\ast \alpha\in H^n_{nr}(k(Y)/k,\mu_m^{\otimes n})=H^n_{nr}(k(Z)/k,\mu_m^{\otimes n})
$$
of order $m$ such that for any subvariety $E\subset Y'$ which does not dominate $\CP^n$ via $ f\circ \tau'$, the class $(\tau')^\ast \gamma$ vanishes in $H^n(k(E),\mu_m^{\otimes n})$.
Moreover, the class $f^\ast \alpha$ vanishes at the generic point of the exceptional divisor $D_0$ of $Y\to Z$, which is cut out by $y_0=0$, 
    and at the generic point of the strict transform $D_{r+1}$ of the divisor $\{y_{r+1} = 0\} \subset Z$ under the blow-up morphism $Y \to Z$.
    Indeed, the divisors $D_0$ and $D_{r+1}$ in $Y$ map via $f$ onto $\CP^n$ and the generic fibres $D_{0,\eta}$ and $D_{r+1,\eta}$ of $\left.f\right|_{D_i}$ are the hypersurfaces in (different) $\CP^r_{k(\CP^n)}$ given by
    $$
    \begin{aligned}
        %D_{0,\eta} &=\left\{
        &\sum\limits_{j = 1}^{r} c_j(x_1,\dots,x_n) y_j^m + (-1)^n x_1 x_2 \dots x_n y_{r+1}^m =0 \quad \text{and} \\ 
        % = 0\right\} \subset \CP^{r}_{k(\CP^n)} \quad \text{and} \\ D_{r+1,\eta} &= \left\{
        &\pi \cdot  \left(1+\sum\limits_{i=1}^n x_i^{\left\lceil \frac{n+1}{m}\right\rceil}\right)^m y_0^m - (-1)^n x_1 x_2 \cdots x_n y_0^m + \sum\limits_{j = 1}^{r} c_j(x_1,\dots,x_n) y_j^m =0, %=0 \right\} \CP^{r}_{k(\CP^n)},
    \end{aligned}
    $$
    respectively, see \eqref{eq:induction-start-first-resolution} and \eqref{eq:gfromSch-torsion}. Thus $D_{0,\eta}$ and $D_{r+1,\eta}$ are each isomorphic as $k(\CP^n)$-varieties to a subvariety of the hypersurface given by the vanishing of the $n$-th Fermat-Pfister form of degree $m$
    $$
        \sum\limits_{\varepsilon \in \{0,1\}^n} (-x_1)^{\varepsilon_1} \dots (-x_n)^{\varepsilon_n} y^m_{j(\varepsilon)}
    $$
    with $j(\varepsilon) = \sum\limits_{i=0}^n \varepsilon_i 2^{i-1}$ as defined in \cite[Equation (7)]{Sch-torsion}. The vanishing of $f^\ast \alpha$ along the generic points of $D_0$ and $D_{r+1}$ follows from \cite[Corollary 4.3]{Sch-torsion}. 
Since the generic fibre of $f$ is smooth ($m$ is invertible in $k$), we conclude via \cite{BO} that 
\begin{align} \label{eq:vanishing-alpha}
((\tau')^\ast \gamma)|_E=0\in H^n(k(E),\mu_m^{\otimes n})
\end{align} 
for any subvariety $E\subset Y'$ with 
$\tau (E)\subset W_Z = Z^{\sing}\cup \{ y_{r+1} l=0\}$, where $l\in k[x_0,\dots ,x_n]$ is any nontrivial homogeneous polynomial in $x_0,\dots ,x_n$ as in the statement of the theorem.

A computation with the Merkurjev pairing similar to  \cite[\S 3]{Sch-JAMS} or \cite[Theorem 8.6]{Sch-survey} then shows $\Tor^{\Z/m}(Z,W_Z)=m$; we sketch the argument for convenience.
For a contradiction assume that there is a positive integer $m'< m$ with
$$
m'\cdot \delta_{Z}=z+m\cdot \zeta \in \CH_0(Z_{k(Z)}) 
$$
for some zero-cycle $z$ whose support $\supp z$ lies in $(W_Z)_{k(Z)}$ and some zero-cycle $\zeta\in \CH_0(Z_{k(Z)})$ that reflects the fact that we work with $\Z/m$-torsion orders.
We restrict the above identity to the regular locus of $Z_{k(Z)}$ and pull this pack to $\tau^{-1}(Z^{\sm}_{k(Z)})$.
The localization sequence then yields
$$
m'\cdot \delta_\tau=z'+m\cdot \zeta'\in \CH_0(Y'_{k(Z)}) , 
$$
where $\delta_\tau=\tau^\ast \delta_Z$ is the point induced by the graph of $\tau:Y'\to Z$, $z'$ is a zero-cycle with $\supp z'\subset \tau^{-1}(W_Z)_{k(Z)}$ and $\zeta'\in \CH_0(Y'_{k(Z)}) $. Note that this used $Z^{\sing}\subset W_Z$.
We pair the above zero-cycle via the Merkurjev pairing (see \cite[\S 2.4]{merkurjev} or \cite[\S 5]{Sch-survey}) with the unramified class $\gamma$ from above.
This yields
$$
\langle m'\cdot \delta_\tau, \tau^\ast \gamma \rangle =0 ,
$$
because $\gamma$ is $m$-torsion, hence pairs to zero with $m\cdot \zeta'$, and it restricts to zero on generic points of subvarieties of $\tau^{-1}(W_Z )$ by \eqref{eq:vanishing-alpha}, hence pairs to zero with $z'$.
Conversely, the definition of the Merkurjev pairing directly implies that 
\begin{align*}
0&=\langle m'\cdot \delta_\tau, \tau^\ast \gamma \rangle= m' \cdot \tau_\ast \tau^\ast \gamma=m'\cdot \deg(\tau)\cdot \gamma \in H^n(k(Z),\mu_m^{\otimes n}),
\end{align*}
 as  $\delta_\tau\in Y'_{k(Z)}$ is the point associated to the graph $\Gamma_\tau\subset Y'\times Z$ of $\tau$. 
This contradicts the fact that $ \deg(\tau)$ is coprime to $m$, that $1\leq m'<m$, and that $\gamma$ has order $m$.
This concludes the proof of the theorem.
\end{proof}

\section{Proof of the main results}

Let $k$ be an algebraically closed field and let $X\subset \CP^{N+1}_k$ be a smooth Fano hypersurface, i.e.\ a smooth hypersurface of degree $d\leq N+1$.
Then $\deg \colon \CH_0(X)\to \Z$ is an isomorphism and so the torsion order $\Tor(X)$ of $X$ is the torsion order $\Tor^\Z(X,W)$ of $X$ relative to any closed zero-dimensional subset $W \subset X$, see Lemma \ref{lem:torsion-order-observations} \eqref{item:4:torsion-order-observations}.
The main results of this paper, stated in the introduction, will follow from the following result. 

\begin{theorem}\label{thm:main-torsion}
    Let $k$ be a field and let $m \geq 2$ be an integer invertible in $k$. 
    Let $n \geq 2$, $r \leq 2^n-2$, and let
    $$   
    s \leq \sum\limits_{l = 1}^n \binom{n}{l} \left\lfloor \frac{n-l}{m} \right\rfloor
    $$
    be non-negative integers. Write $N\coloneq  n+r+s$.
    Then the torsion order $\Tor(X_d)$ of a very general Fano hypersurface $X_d \subset \CP^{N+1}_k$ of degree $d \geq m + n$ is divisible by $m$.
\end{theorem}

\begin{remark}
    For $s=0$, the result is proven in \cite[Theorem 7.1]{Sch-torsion}.
\end{remark}

\begin{proof}[Proof of Theorem \ref{thm:main-torsion}]
Note that the torsion order of any variety is divisible by the torsion order of the base-change to any field extension.  
Moreover, the definition of very general (see Definition \ref{def:very-general}) is stable under extension of the base field.
Up to replacing $k$ by a field extension, we can therefore assume that $k$ is algebraically closed and uncountable. 

We fix positive integers $n \geq 2$ and $r \leq 2^{n} -2$. 
For a non-negative integer $s$ as in the theorem, we define inductively an integral degree $d$-hypersurface $Z = Z_s$ of dimension $N = n + r + s$. 

\textbf{Step 1.}
Suppose there exists an integral degree $d$-hypersurface $Z_s \subset \CP_k^{n+r+s+1}$ given by the vanishing of a homogeneous polynomial of the form
    \begin{equation}\label{eq:Zs}
        f_0^{(s)} + a_0^{(s)} +  a_{r+1}^{(s)} y_{r+1}^m + \sum\limits_{j=1}^{r} \sum\limits_{i=1}^m a_{i,j}^{(s)} \cdot y_j^i \in k[x_0,\dots,x_n,y_1,\dots,y_{r+1},z_1,\dots,z_s]
    \end{equation}
    for some homogeneous polynomials $$
        f_0^{(s)}, a_0^{(s)}, a_{r+1}^{(s)}, a_{i,j}^{(s)} \in k[x_0,\dots,x_n,z_1,\dots,z_s]
    $$
    such that \begin{enumerate}[(1)]
        \item $f_0^{(s)} + a_0^{(s)} \in k[x_0,\dots,x_n,z_1,\dots,z_s]$ is an irreducible polynomial of degree $d$; \label{item:PfTorsionOrder1} 
        \item for each $j$, if $x_0^{e m} \mid a_{m,j}^{(s)}$ for a non-negative integer $e$, then $x_0^{e i} \mid a_{i,j}^{(s)}$ for all $1 \leq i \leq m$. We denote the maximal such $e$ by $e_j^{(s)}$; \label{item:PfTorsionOrder2}
        \item $\Tor^{\Z/m}(Z_s,W_{s}) = m$, where $W_s \coloneq  \{x_0 a_{r+1}^{(s)} y_{r+1} h^{(s)} = 0\} \cup Z_s^{\sing} \subset Z_s$ for some homogeneous polynomial $h^{(s)} \in k[x_0,\dots,x_n,z_1,\dots,z_s]$. \label{item:PfTorsionOrder3}
    \end{enumerate}
    Assume that there exists some $1 \leq j_0 \leq r$ such that $e_{j_0}^{(s)} \geq 1$. 
    Then we construct an integral degree $d$-hypersurface $Z_{s+1}$ of the same form \eqref{eq:Zs} satisfying the condition \eqref{item:PfTorsionOrder1}, \eqref{item:PfTorsionOrder2}, and \eqref{item:PfTorsionOrder3} as follows. 
    
    Consider $k' = \overline{k(\lambda)}$ an algebraic closure of the purely transcendental field extension $k(\lambda)$ of $k$. 
    Rewrite the equation \eqref{eq:Zs} as
    $$
        \left(f_0^{(s)} + a_{r+1}^{(s)} y_{r+1}^m + \sum\limits_{j\neq j_0} \sum\limits_{i=1}^m a_{i,j}^{(s)} \cdot y_j^i\right) + a_0^{(s)} + \sum\limits_{i=1}^m a_{i,j_0}^{(s)} y_{j_0}^i.
    $$
    We are now in the situation of Section \ref{section:doublecone}.
    Note that condition \eqref{item:PfTorsionOrder1} implies the assumption \eqref{eq:conditionirreducible}.
    Let $Z_{s+1}$ be the integral degree $d$-hypersurface $X' \times_{K} \Bar{K} \subset \CP^{n+r+s+2}_{\Bar{K}}$, where $X'$ is as in Lemma \ref{lem:genericfibre} and $K = k'(t)$. We choose an isomorphism $k \cong \Bar{K}$, which exists because both fields are algebraically closed, have the same characteristic and have the same  uncountable transcendence degree over their prime fields.
    Thus we can view $Z_{s+1}$ as a variety over $k$. 
    We aim to check that $Z_{s+1}$ satisfies the assumptions above. Recall from Lemma \ref{lem:genericfibre} that $Z_{s+1}\subset \CP^{n+r+s+2}_k$ is cut out by the homogeneous polynomial
    \begin{equation} \label{eq:induction-step-polynomial}
        \left(f_0^{(s)} + a_{r+1}^{(s)} y_{r+1}^m + \sum\limits_{j\neq j_0} \sum\limits_{i=1}^m a^{(s)}_{i,j} \cdot y_j^i\right) + a^{(s+1)}_0 + \sum\limits_{i=1}^m a^{(s+1)}_{i,j_0} y_{j_0}^i,
    \end{equation}
    where $a^{(s+1)}_0,a^{(s+1)}_{i,j_0}\in k[x_0,\dots,x_n,z_1,\dots,z_{s+1}]$ are defined as in \eqref{eq:aiprime}. 
    Note this uses also Lemma \ref{lem:genericfibre} \eqref{item:lem:genericfibre3}.
    Hence, the defining equation of $Z_{s+1}$ has the form \eqref{eq:Zs} with
    $$
    f_0^{(s+1)} \coloneq  f_0^{(s)}, \quad a_{r+1}^{(s+1)} \coloneq  a_{r+1}^{(s)}, \quad \text{and} \quad a^{(s+1)}_{i,j}\coloneq a^{(s)}_{i,j}\ \ \ \text{for $j\neq j_0$}.
    $$ %where we used Lemma \ref{lem:genericfibre} \eqref{item:lem:genericfibre3}.
    Since $f_{0}^{(s+1)}$ and $a_0^{(s+1)}$ do not contain any $y_i$, condition \eqref{item:PfTorsionOrder1} follows from Lemma \ref{lem:genericfibre} \eqref{item:lem:genericfibre4}.
    As $a_{i,j}^{(s+1)} = a_{i,j}^{(s)}$ for $j \neq j_0$, condition \eqref{item:PfTorsionOrder2} is clearly satisfied for $j \neq j_0$ and we note that $e_j^{(s+1)} = e_j^{(s)}$ for $j \neq j_0$. For $j = j_0$, condition \eqref{item:PfTorsionOrder2} follows from Lemma \ref{lem:genericfibre} \eqref{item:lem:genericfibre2}. 
    Proposition \ref{prop:inductionstep} shows that $\Tor^{\Z/m}(Z_{s+1},W_{s+1}) = m$, where $W_{s+1} = \{x_0 a_{r+1}^{(s)} y_{r+1} h^{(s)} z_{s+1} = 0\} \subset Z_{s+1}$.  
    Note that the derivative of \eqref{eq:induction-step-polynomial} with respect to $y_{r+1}$ is equal to $m a_{r+1}^{(s)} y_{r+1}^{m-1}$;
    thus the singular locus of $Z_{s+1}$ is contained in $\{a_{r+1}^{(s)} y_{r+1}=0\} \subset Z_{s+1}$.
    By definition, $a_{r+1}^{(s+1)} = a_{r+1}^{(s)}$ and so condition \eqref{item:PfTorsionOrder3} is satisfied for $h^{(s+1)} = h^{(s)} z_{s+1}$. 

\textbf{Step 2.}
    Consider the hypersurface $Z_0 \coloneq  Z_\rho$ with defining equation $\rho h + x_0^{d-m-n} F$ as in Corollary \ref{cor:Induction-Start-3}, where $h \in k[x_0,\dots,x_n]$ is an irreducible polynomial of degree $d$, e.g.
    $$
        h = \begin{cases}
            x_0^d + \sum\limits_{i = 1}^n x_{i-1} x_i^{d-1} & \text{if } p > 0 \text{ and } p \mid d, \\
            \sum\limits_{i = 0}^n x_i^d & \text{otherwise,}
        \end{cases}
    $$
    where $p$ is the characteristic of $k$.
    We will prove that $Z_0$ satisfies the condition \eqref{item:PfTorsionOrder1}, \eqref{item:PfTorsionOrder2}, and \eqref{item:PfTorsionOrder3} above. 

    Consider the following polynomials in $k[x_0,\dots,x_n]$ 
    \begin{align*}
        f_0^{(0)} &\coloneq  \rho h + x_0^{d-\deg(g)} g, %\\
       \quad   a_0^{(0)} \coloneq  0, \quad
         a_{r+1}^{(0)}  \coloneq  (-1)^n x_0^{d-m-n} x_1 x_2 \dots x_n, \\ 
        a_{i,j}^{(0)} &\coloneq  0, \quad \text{for } 1 \leq i \leq m-1 \text{ and } 1 \leq j \leq r, \\
        a_{m,j}^{(0)} &\coloneq  x_0^{d-m-\deg(c_j)} c_j(x_1,\dots,x_n), \quad \text{for } 1 \leq j \leq r,
    \end{align*}
    where $g$ is defined in \eqref{eq:gfromSch-torsion} and the $c_j$'s are defined in \eqref{eq:FermatPfister}. Then, by construction, 
    $$
        \rho h + x_0^{d-m-n} F = f_0^{(0)} + a_0^{(0)} +  a_{r+1}^{(s)} y_{r+1}^m + \sum\limits_{j=1}^r \sum\limits_{i=1}^m a_{i,j}^{(0)} \cdot y_j^i \in k[x_0,\dots,x_n,y_1,\dots,y_{r+1}]
    $$
    is of the form \eqref{eq:Zs}. 
    The polynomial $f_0^{(0)}$ is irreducible because the polynomial $h$ is irreducible and $\rho$ is a transcendental parameter over the prime field of $k$, which is algebraically independent from $\pi$.
    Hence condition \eqref{item:PfTorsionOrder1} holds.
    Condition \eqref{item:PfTorsionOrder2} is clearly satisfied as $a_{i,j}^{(0)} = 0$ for $1 \leq i \leq m-1$. 
    In particular, we see from the definition of $a_{m,j}^{(0)}$ above that \begin{equation}
        e_j^{(0)} = \left\lfloor \frac{d-m-\deg(c_j)}{m} \right\rfloor.
    \end{equation}
    Corollary \ref{cor:Induction-Start-3} shows that condition \eqref{item:PfTorsionOrder3} holds as well, where we note that we can choose $h^{(0)} = 1$.
    This concludes Step 2.

By Steps 1 and 2 above, we can apply the double cone construction (see Section \ref{section:doublecone}) as long as at least one of the $e_j$'s defined in \eqref{item:PfTorsionOrder2} is positive. In each step, we reduce one of them by $1$. Hence, the number of steps is equal to the sum
    $$
        \sum\limits_{j=1}^{r} e_j^{(0)} = \sum\limits_{j=1}^{r} \left\lfloor \frac{d-m-\deg(c_j)}{m} \right\rfloor.
    $$
    This sum becomes maximal when $r$ is maximal, i.e. $r = 2^n -2$. Then the sum reads
    \begin{equation*}
        \sum\limits_{j=1}^{2^n-2} e_j^{(0)} = \sum\limits_{l=1}^{n-1} \binom{n}{l} \left\lfloor \frac{d-m-l}{m} \right\rfloor.
    \end{equation*}

    Let now $X_d$ be a very general hypersurface of degree $d$ and dimension $N$ over $k$.
    Up to replacing $k$ by a larger algebraically closed field (which does not affect the torsion order by Lemma \ref{lem:torsion-order-observations}), we can by Lemma \ref{lem:very-general} assume that $X_d$ degenerates to $Z_s$. 
    By choosing $N$ general hyperplane sections through a closed point of $W_s$, we can assume that there exists a closed subset $W$ in the total space of the degeneration which has relative dimension $0$ and whose restriction to the special fibre $Z_s$ is contained in $W_s$. 
    Applying Lemma \ref{lem:Lambda-torsion-order-specialize} to the degeneration with $W$ as closed subset yields that
    $$
        \Tor^{\Z/m}(X_d,W_{X_d}) = m,
    $$
    where $W_{X_d} \coloneq  W \cap X_d \subset X_d$ is a closed non-empty zero-dimensional subset of $X_d$. Thus, Lemma \ref{lem:torsion-order-observations} implies that $m$ divides $\Tor(X_d)$, which finishes the proof of the theorem.
\end{proof}

The following lemmas yield explicit estimates for the bound given in Theorem \ref{thm:main-torsion}.

\begin{lemma}\label{lem:bound}
    Let $n,m \geq 2$ be positive integers. Then
    \begin{equation}\label{eq:estimate}
         \left(\left\lfloor\frac{n}{m}\right\rfloor - 1 \right) (2^{n-1}-1) \leq \sum\limits_{l=1}^n \binom{n}{l} \left\lfloor \frac{n-l}{m} \right\rfloor \leq \left\lfloor\frac{n}{m}\right\rfloor (2^{n-1}-1).
    \end{equation}
\end{lemma}

\begin{proof}
The sum in question can be rewritten as follows
    \begin{align*}
        \sum\limits_{l=1}^n \binom{n}{l} \left\lfloor \frac{n-l}{m} \right\rfloor &= \frac{1}{2} \sum\limits_{l=1}^{n-1} \binom{n}{l} \left\lfloor \frac{n-l}{m} \right\rfloor + \frac{1}{2} \sum\limits_{l=1}^{n-1} \binom{n}{n-l} \left\lfloor \frac{n-l}{m} \right\rfloor \\
        &= \frac{1}{2} \sum\limits_{l=1}^{n-1} \binom{n}{l} \left\lfloor \frac{n-l}{m} \right\rfloor + \frac{1}{2} \sum\limits_{l=1}^{n-1} \binom{n}{l} \left\lfloor \frac{l}{m} \right\rfloor \\
        &= \frac{1}{2} \sum\limits_{l=1}^{n-1} \binom{n}{l} \left(\left\lfloor \frac{n-l}{m} \right\rfloor + \left\lfloor \frac{l}{m} \right\rfloor\right).
    \end{align*}
    The estimates in \eqref{eq:estimate} follow now from the observation $$
        \left(\left\lfloor\frac{n}{m}\right\rfloor - 1 \right) \leq \left(\left\lfloor \frac{n-l}{m} \right\rfloor + \left\lfloor \frac{l}{m} \right\rfloor\right) \leq \left\lfloor \frac{n}{m} \right\rfloor
    $$
    for all $0 \leq l \leq n$ together with the summation formula for binomial coefficients.
\end{proof}

We provide more explicit formulas for $m = 2$ and $m = 3$. These bounds are used
in Theorem \ref{thm:main-irrationality-intro} and Theorem \ref{thm:main-irrationality-intro-char2}.

\begin{lemma}\label{lem:bound-for2and3}
    Let $n$ be a positive integer. Then the following formulas hold
    \begin{align}
        \sum\limits_{l=1}^n \binom{n}{l} \left\lfloor \frac{n-l}{2} \right\rfloor &= (n-1) 2^{n-2} - \left\lfloor \frac{n}{2} \right\rfloor, \label{eq:bound-2} \\
        \sum\limits_{l=1}^n \binom{n}{l} \left\lfloor \frac{n-l}{3} \right\rfloor &= \frac{n-2}{3} 2^{n-1} - \frac{n}{3} + \delta, \label{eq:bound-3}
    \end{align}
    where $\delta$ depends on the remainder of $n$ modulo $6$ and is given by the following table.
    $$
        \begin{array}{l|c|c|c|c|c|c}
            n \mod 6& 0 & 1 & 2 & 3 & 4 & 5  \\ \hline
            \delta & \frac{1}{3} & \frac{2}{3} & \frac{2}{3} & -\frac{1}{3} & 0 & \frac{2}{3}.
        \end{array}
    $$
\end{lemma}

\begin{proof}
    We check \eqref{eq:bound-2} by the following computation:
    \begin{align*}
        \sum\limits_{l=1}^n \binom{n}{l} \left\lfloor \frac{n-l}{2} \right\rfloor &= \sum\limits_{l=1}^n \frac{n-l}{2} \binom{n}{l} - \frac{1}{2}\sum\limits_{\substack{l=1 \\ n -l \text{ odd}}}^n \binom{n}{l} \\
        &= \frac{n}{2} (2^{n} - 1) - \frac{1}{2}\sum\limits_{l=1}^n l \binom{n}{l} - \frac{1}{2} \sum\limits_{\substack{l=1 \\ n-l \text{ odd}}}^n \left[\binom{n-1}{l-1} + \binom{n-1}{l}\right] \\
        &= \frac{n}{2} (2^{n} - 1) - n 2^{n-2} - 2^{n-2} + \left(\frac{1}{4} - (-1)^n \frac{1}{4}\right) \\
        &= (n-1)2^{n-2} - \left\lfloor \frac{n}{2} \right\rfloor.
    \end{align*}
    We turn to \eqref{eq:bound-3} and prove first a combinatorial formula for a lacunary sum of binomial coefficients, see e.g. \cite[Section 4, Problem 8, p.161]{Riordan}. Let $\xi$ be a primitive third root of unity in $\C$, then the following holds
    $$
        \begin{aligned}
            3 \sum\limits_{l \equiv r \ (3)} \binom{n}{l} &= \sum\limits_{l = 0}^n \binom{n}{l} (1 + \xi^{l-r} + \xi^{{2l-2r}}) \\
            &= 2^n + \xi^{-r}(1+\xi)^{n} + \xi^{-2r} (1+\xi^2)^n \\
             &= 2^n + (-1)^n (\xi^{n+r})^2 + (-1)^n\xi^{n+r}.
        \end{aligned}
    $$
    Write $n = 3a + b$ for integers $a \in \Z_{\geq 0}$ and $b \in \{0,1,2\}$. Then we get
    $$
    \begin{aligned}
        \sum\limits_{l=1}^n \binom{n}{l} \left\lfloor \frac{n-l}{3} \right\rfloor &= \sum\limits_{l=1}^n \frac{n-l}{3} \binom{n}{l} - \frac{1}{3}\sum\limits_{\substack{n-l \equiv 1 \ (3) \\ l \neq 0}} \binom{n}{l} - \frac{2}{3}\sum\limits_{\substack{n-l \equiv 2 \ (3) \\ l \neq 0}} \binom{n}{l}  \\
        &= \frac{n}{3} (2^{n-1}-1) - \frac{1}{3}\sum\limits_{l \equiv 1 \ (3)} \binom{n}{l} - \frac{2}{3}\sum\limits_{l \equiv 2 \ (3)} \binom{n}{l}  + \frac{b}{3}\\
        &= \frac{n-2}{3} 2^{n-1} - \frac{n}{3} + \frac{b}{3} - \frac{(-1)^n}{9} \left(\xi^{2b+2} + \xi^{b+1} + 2 \xi^{2b+1} + 2 \xi^{b+2}\right). \\
    \end{aligned}
    $$
    A simple computation shows that
    $$
        \frac{b}{3} - \frac{(-1)^n}{9} \left(\xi^{2b+2} + \xi^{b+1} + 2 \xi^{2b+1} + 2 \xi^{b+2}\right) = \delta,
    $$
    which proves \eqref{eq:bound-3} and thus the lemma.
\end{proof}

\begin{proof}[Proof of Theorem \ref{thm:main-irrationality-intro}]
    If $d = 4$, then the statement follows from \cite{totaro} for $N \leq 4$ (see also \cite[Theorem 1.1]{Sch-JAMS}) and \cite[Theorem 1.1]{PavicSch} for $N = 5$, see also
    Example \ref{ex:quartic5folds}.
    Let now $d \geq 5$ and $N \leq \frac{d+1}{16} 2^d$ be positive integers. If $3 \leq N \leq (d-2) + 2^{d-2}-2$, then the theorem follows from \cite[Theorem 1.1]{Sch-JAMS}, because we can then write $N$ uniquely as $N = n + r$ for integers $n,r \geq 1$ satisfying $2 \leq n \leq d-2$ and $2^{n-1}-2 \leq r \leq 2^n - 2$. Hence we can assume that $N \geq d-4+2^{d-2}$. Let $n = d-2$, $r = 2^{d-2} - 2$ and $s = N - n - r$ be non-negative integers. We claim that $$
        s \leq (d-3) 2^{d-4} - \left\lfloor \frac{d-2}{2} \right\rfloor.
    $$
    Indeed, otherwise we get
    $$
        N = n+r+s > d-4 + 2^{d-2} + (d-3) 2^{d-4} - \left\lfloor \frac{d-2}{2} \right\rfloor \geq \frac{d+1}{16}2^d,
    $$
    which yields a contradiction to the assumption $N \leq \frac{d+1}{16} 2^d$. Thus, Theorem \ref{thm:main-torsion} implies by Lemma \ref{lem:bound-for2and3} that $\Tor(X_d)$ is divisible by $2$ for a very general degree $d$ hypersurface $X_d \subset \CP^{N+1}$. In particular, $X_d$ does not admit an integral decomposition of the diagonal.
\end{proof}

\begin{proof}[Proof of Theorem \ref{thm:main-irrationality-intro-char2}]
    Let $d \geq 5$ and $3 \leq N \leq \frac{d+1}{48} 2^d$ be integers. If $3 \leq N \leq (d-3) + 2^{d-3} - 2$, then the theorem follows from \cite[Theorem 7.1]{Sch-torsion}, because we can write $N = n+r$ for unique positive integers $n,r$ satisfying $2 \leq n \leq d-3$ and $2^{n-1}-2 \leq r \leq 2^n -2$. Hence we can assume that $N \geq d-5 + 2^{d-3}$. Let $n = d-3$, $r = 2^{d-3}-2$ and $s = N -n-r$ be non-negative integers. We claim that $$
        s \leq \frac{d-5}{3} 2^{d-4} - \frac{d-3}{3} + \delta,
    $$
    where $\delta$ is defined as in Lemma \ref{lem:bound-for2and3}. (Note that $n = d-3$.) Indeed, otherwise we get
    $$
        N = n + r + s > d-3 + 2^{d-3}-2 + \frac{d-5}{3} 2^{d-4} - \frac{d-3}{3} + \delta = \frac{d+1}{48} 2^d + \frac{2d}{3} + \delta - 4 \geq \left\lfloor\frac{d+1}{48} 2^d\right\rfloor,
    $$
    which yields a contradiction to the assumption that $N$ is an integer satisfying $N \leq \frac{d+1}{48}2^d$. Thus, Theorem \ref{thm:main-torsion} implies by Lemma \ref{lem:bound-for2and3} that $3 \mid \Tor(X_d)$ for a very general degree $d$ hypersurface $X_d \subset \CP^{N+1}$. In particular, $X_d$ does not admit an integral decomposition of the diagonal.
\end{proof}

\begin{proof}[Proof of Theorem \ref{thm:main-torsion-order-intro}]
    This follows from Theorem \ref{thm:main-torsion} together with the estimate in Lemma \ref{lem:bound}.
\end{proof}

\section*{Acknowledgement}
The first named author is supported by the Studienstiftung des deutschen Volkes.
We are grateful to Simen Moe for discussions and to Jean-Louis Colliot-Thélène and the anonymous referee for comments that improved the exposition. 
This project has received funding from the European Research Council (ERC) under the European Union's Horizon 2020 research and innovation programme under grant agreement No 948066 (ERC-StG RationAlgic).

%\begin{thebibliography}{9}  

\end{document}